\pdfoutput=1
\documentclass[preprint,12pt]{elsarticle}

\usepackage{graphicx}
\usepackage{amsmath,amsthm,amssymb,scrextend,bbm}
\usepackage{mathtools}

\usepackage{xcolor}
\usepackage{epstopdf}
\usepackage{ascii}

\journal{Journal of Computational Physics}
\DeclarePairedDelimiter\norm{\lVert}{\rVert}%

\newtheorem{theorem}{Theorem}
\newtheorem{proposition}[theorem]{Proposition}
\newtheorem{lemma}[theorem]{Lemma}

\newcommand{\R}{\mathbb{R}}

\newcommand{\lep}{{L_\epsilon}}
\newcommand{\alep}{{L^*_\epsilon}}
\newcommand{\ko}{\mathcal{L}}
\newcommand{\al}{\mathcal{L}^*}

\newcommand{\BEA}{\begin{eqnarray}}
\newcommand{\EEA}{\end{eqnarray}}
\newcommand{\comment}[1]{}

\newcommand{\gq}{G_{q,\epsilon}}
\newcommand{\te}{\tilde{\epsilon}}
\newcommand{\tm}{\tilde{m}}
\newcommand{\tmo}{\tilde{m}_0}
\newcommand{\tmi}{\tilde{m}_1}
\newcommand{\tom}{\tilde{\omega}}
\newcommand{\mt}{\mathcal{O}(\tilde{\epsilon}^2)}
\newcommand{\me}{\mathcal{O}(\epsilon^2)}
\newcommand{\qt}{q_{\tilde{\epsilon}}}

\begin{document}

\title{Approximating solutions of linear elliptic PDE's on a smooth manifold using local kernel}
\author[rvt]{Faheem Gilani\corref{cor}}
\ead{fhg3@psu.edu}
\author[rvt,rvt2,rvt3]{John Harlim}
\ead{jharlim@psu.edu}
\cortext[cor]{Corresponding author}
\address[rvt]{Department of Mathematics, the Pennsylvania State University, 109 McAllister Building, University Park, PA 16802-6400, USA}
\address[rvt2]{Department of Meteorology and Atmospheric Science, the Pennsylvania State University, 503 Walker Building, University Park, PA 16802-5013, USA}
\address[rvt3]{Institute for CyberScience, the Pennsylvania State University, 224B Computer Building, University Park, PA 16802, USA}
\date{\today}

\begin{abstract}
A mesh-free numerical method for solving linear elliptic PDE's using the local kernel theory that was developed for manifold learning is proposed. In particular, this novel approach exploits the local kernel theory which allows one to approximate the Kolmogorov operator associated with It\^o diffusion processes on compact Riemannian manifolds without boundary or with Neumann boundary conditions using an integral operator. Theoretical justification for the convergence of this numerical technique is provided under the standard conditions for the existence of the weak solutions of the PDEs. Numerical results on various instructive examples, ranging from PDE's defined on flat and non-flat manifolds with known and unknown embedding functions show accurate approximation with error on the order of the kernel bandwidth parameter.
\end{abstract}


\maketitle

\begin{section}{Introduction}\label{introduction}

An important classical model in applied mathematics is the second order elliptic linear partial differential equations (see e.g., \cite{evans1998partial}). This well-studied boundary value problem arises in various applications including fluid flow, elasticity, electromagnetism, heat conduction \cite{Feynman}, neutron diffusion \cite{Neutron}, and probability theory \cite{MoPer}.
 For PDE's on manifolds (especially on two-dimensional surfaces), many methods have been developed to numerically approximate a solution. 
While the implementation detail of each of the existing methods is different in its own way, most of them have a unifying theme: they require a representation of the surface to approximate the tangential derivatives along the surface. For example, the finite element method (FEM) uses (triangulated) meshes to approximate the surface \cite{dziuk2013finite,camacho,bonito2016high}. The approach in \cite{bertalmio2001variational,memoli2004implicit} represents the surface using level sets. The closest point method \cite{ruuth2008simple} uses a closest point representation of the surface. The mesh-free radial basis function (RBF) method represents the surface using a distance function such that the surface is a level set of the function \cite{piret2012orthogonal}. See also the references in \cite{piret2012orthogonal,li2016convergent} for a more comprehensive literature review on this topic. For a high-dimensional manifold $\mathcal{M}$ embedded in $\mathbb{R}^n$, obtaining these manifold representations from the ambient data (or the so-called point cloud in this community) can be challenging, as pointed out in \cite{li2016convergent}.

In data science, the problem of characterizing a manifold from the embedded data is known as manifold learning. 
One of the most popular and theoretically rigorous approaches to manifold learning is the diffusion maps algorithm \cite{cl:06,bh:15vb}. In a nutshell, the diffusion maps algorithm characterizes the manifold using the eigenfunctions of the Laplace-Beltrami (or weighted Laplacian) operator estimated from the data that lie on the manifold. Technically, the diffusion maps algorithm approximates the Laplacian operator using a local integral operator with exponentially decaying kernel functions defined on the ambient data (or point cloud). In this paper, our aim is to leverage this operator estimation technique to solve PDE's on smooth manifolds without boundary or with Neumann boundary conditions imposed on the solution when the manifold has a boundary. A closely related method that shares the same idea is the point integral method (PIM) for solving Poisson problems \cite{li2017point} and isotropic elliptic equations \cite{li2016convergent}. The proposed approach in this paper can be interpreted as a generalization of the PIM to non-symmetric diffusion operators of Kolmogorov type which leverages the recently developed local kernel theory \cite{bs:16}. A detailed discussion about the connection of the proposed approach to PIM is presented in Section~2.

In particular, we consider the following boundary value problem:
\BEA
\begin{cases}
(a+\mathcal{L})u(x) &= f(x), \ \ \ \ x  \in \mathcal{M}  \\
\partial_{\nu}u\mid_{\partial \mathcal{M}}&=0,
 \label{PDE}
\end{cases}
\EEA
 where $\mathcal{M}\subseteq\mathbb{R}^n$ is a compact $d$-dimensional smooth manifold, embedded in $\mathbb{R}^n$. We note that the proposed method is also valid for manifolds without boundary. In \eqref{PDE}, the term $a$ (by an abuse of notation) denotes the multiplication operator by a scalar function $a: \mathcal{M}\rightarrow \R$, and the differential operator $\mathcal{L}$ is the backward Kolmogorov operator of It\^o diffusion,  
\BEA
\mathcal{L} =  b \cdot \nabla + \frac{1}{2} c_{ij} \nabla_i\nabla_j,   \label{generalL}
\EEA
where $\nabla$ denotes the gradient, $\nabla_i$ denotes the covariant derivative in the $i$th direction, and $\nabla_i\nabla_j$ denotes the components of the Hessian operator. Here the differential operators and the dot product are defined with respect to the Riemannian metric inherited by $\mathcal{M}$ from $\mathbb{R}^n$. The differential operator $\mathcal{L}$ involves a vector field $b:\mathcal{M} \to \mathbb{R}^d$ and a symmetric positive definite diffusion tensor $c:\mathcal{M} \to \mathbb{R}^d\times \mathbb{R}^d$. The main idea in this paper is to apply the integral operator estimation from the local kernel theory \cite{bs:16} to solve the boundary value problem described in \eqref{PDE}. To formalize this mesh-free scheme, we will discuss the well-posedness of the approximate linear problem and the convergence of the solution operator.

The remainder of this paper is organized as follows: In Section~2, we review the local kernel theory, formulate our approach, and compare it to the point integral method (PIM). In Section~3, we discuss the convergence of the proposed local integral approximation, the numerical discretization and its convergence rate. In Section~4, we provide numerical demonstrations of the method on several examples. Comparisons of the proposed scheme with RBF and FEM will be shown. Finally, we close the paper with a short summary in Section~5.
\end{section}

\section{Approximating the differential operator with a local integral operator}

In this section, we briefly review the relevant results from \cite{bh:15vb} that serve as the foundation for the numerical approach proposed in this paper. In particular, we review an asymptotic expansion that allows one to approximate the differential operator $\mathcal{L}$ (as well as $\mathcal{L}^*$) with a local integral operator. 

Let $\mathcal{M}\subseteq\mathbb{R}^n$ be a $d-$dimensional manifold embedded in $\R^n$ via the embedding function $\iota:\mathcal{M}\to\mathbb{R}^n$. For simplicity we abuse the notation $x\in\mathcal{M}$ (instead of using $\iota^{-1}(x)$) to denote points on the manifold with ambient coordinate representation $x\in\mathbb{R}^n$, but we will clarify which coordinates we are referring to in each of the following definition below. We define the prototypical local kernel, $K:\R^+\times \R^n \times \R^n\rightarrow \R$, as,
\BEA
K(\epsilon,x,y)=\exp\left(-\frac{(x-y+\epsilon B(x))^\top C(x)^{-1}(x-y+\epsilon B(x))}{2\epsilon}\right),\label{prototypicalkernel}
\EEA
with a symmetric positive definite $C$. By local kernel, we mean that there exists constants $\beta, \sigma>0$ and a vector field $B:\R^n\rightarrow \R^n$ independent of $\epsilon$ such that 
\BEA
0 \leq K(\epsilon,x,x+\sqrt{\epsilon}z) \leq \beta e^{-\sigma \norm{z-\sqrt{\epsilon}B(x)}^2}, \label{localkernelbound}
\EEA
for all $x,z\in \R^n$.  The first inequality is clear since the kernel is an exponential function. The second inequality can be deduced as follows: since all eigenvalues of the matrix $C$ are real and strictly positive, we have that $\|C(x)^{-1}\|_2 \geq \lambda_1^{-1}$, where $\lambda_1$ denotes the largest eigenvalue of $C$. Thus the estimate in \eqref{localkernelbound} holds for any $\beta\geq 1$ and $\sigma \leq  (2\lambda_1)^{-1}$. 

It is clear from the above that the prototypical kernel in \eqref{prototypicalkernel} evaluates points on the manifold in the ambient coordinate, $x,y \in\R^n$. Let $B:\R^n\rightarrow \R^n$  and $C:\R^n\rightarrow \R^n\times \R^n $ satisfy
\BEA
B(x) &=& (D\iota(x)^\dagger)^\top b(x), \label{Bb}\\
C(x)^{-1} &=& ( D\iota(x) c(x) D\iota(x)^\top)^\dagger,\label{Cc}
\EEA
where the Jacobian $D\iota(x)$ is a map that takes vectors in $T_x \mathcal{M}\cong\R^d$ to $T\R^n\cong\R^n$ and $\dagger$ denotes the pseudo-inverse. Then the zeroth, first, and second moments of the kernel are given through the following limits:
\begin{align*}
m(x)&=\lim_{\epsilon \rightarrow 0}\int_{T_x \mathcal{M}}K(\epsilon,x,x+\sqrt{\epsilon} \hat{z}) \ dz\\
 b_i(x)m(x)&=\lim_{\epsilon \rightarrow 0}\frac{1}{\sqrt{\epsilon}}\int_{T_x \mathcal{M}}z_iK(\epsilon,x,x+\sqrt{\epsilon} \hat{z} )\ dz, \quad \quad i=1,\ldots, d\\
c_{ij}(x)m(x)&=\lim_{\epsilon \rightarrow 0}\int_{T_x \mathcal{M}}z_iz_jK(\epsilon,x,x+\sqrt{\epsilon} \hat{z}) \ dz, \quad\quad i,j=1,\ldots, d.
\end{align*}
Here, notice the abuses of notation in $m, b,$ and $c$ which are functions of intrinsic coordinates as denoted in Section~\ref{introduction} i.e. $m(x) := m(\iota^{-1}(x))$, $b(x) := b(\iota^{-1}(x))$, and $c(x) := c(\iota^{-1}(x))$. In these limits, vector $\hat{z}=(z,0) \in \mathbb{R}^n$ is equal to $z\in\mathbb{R}^d$ onto $T_x\mathcal{M}$ and $0$ in the orthogonal directions. For the prototypical kernel defined in \eqref{prototypicalkernel}, one can deduce that the normalization constant $m(x) =(2\pi)^{d/2}\mbox{det}(c(x))^{1/2}$.

The main result from \cite{bh:15vb} used in this paper is that for any $u\in C^3(\mathcal{M})$, the asymptotic expansion,
\BEA
G_\epsilon u(x) &:=& \epsilon^{-d/2}\int_{\mathcal{M}}K(\epsilon,x,y) u(y) \ dy \nonumber\\ &=& m(x)u(x)+\epsilon(\omega(x)u(x)+m(x)\mathcal{L}u(x))+\mathcal{O}(\epsilon^2), 
\label{gef}\EEA
holds when $\mathcal{M}$ has no boundaries or when $\mathcal{M}$ has boundary and $f$ satisfies the Neumann boundary conditions  $\partial_\nu f |_{\partial \mathcal{M}} = 0$, where $\partial_{\nu}$ is the normal derivative.

From \eqref{gef}, we can approximate the normalization constant $m(x)$ using
\BEA
G_\epsilon 1(x) = \epsilon^{-d/2}\int_{\mathcal{M}}K(\epsilon,x,y) \ dy = m(x) + \epsilon \omega(x)+\mathcal{O}(\epsilon^2).\label{ge1}
\EEA
For notational convenience, we define a normalized kernel:
\BEA
S(\epsilon,x,y):=\frac{K(\epsilon,x,y)}{\int_{\mathcal{M}}K(\epsilon,x,y) \ dy} = \epsilon^{-d/2} \frac{K(\epsilon,x,y)} {G_\epsilon 1(x) }.\label{khat}
\EEA 
Using \eqref{gef} and \eqref{ge1}, we can show that
\BEA
J_\epsilon u(x) &:=& \int_{\mathcal{M}} S(\epsilon,x,y) u(y)\ dy = \frac{G_\epsilon u(x)}{G_\epsilon 1(x)} \nonumber \\
&=& u(x)  + \epsilon \mathcal{L}u(x) + \mathcal{O}(\epsilon^2),\label{ghat}
\EEA
which suggests that the differential operator $\mathcal{L}$ can be pointwise approximated by the following Fredholm integral operator of the second kind:
\BEA
L_\epsilon u(x) &:=& \frac{1}{\epsilon}(J_\epsilon - \mathcal{I}) u(x) = \mathcal{L}u(x) + \mathcal{O}(\epsilon).\label{Le}
\EEA
Here (and in the remainder of this paper), the notation $\mathcal{I}$ denotes an identity operator.
The adjoint operator

\BEA
\mathcal{L}^* =  -\mbox{div} (b \,\cdot)  + \frac{1}{2} \nabla_i\nabla_j (c_{ij} \,\cdot), \label{generalLadj}
\EEA
can also be estimated using almost the same procedure. In particular, one can show that the adjoint of \eqref{gef} yields

\BEA
G^*_\epsilon u(x) &:=&\epsilon^{-d/2}\int_{\mathcal{M}}K(\epsilon,y,x) u(y) \ dy \nonumber\\ &=& m(x)u(x)+\epsilon(\omega(x)u(x)+\mathcal{L}^*(m(x)u(x)))+\mathcal{O}(\epsilon^2), \label{aex}
\EEA

and  \eqref{khat}, \eqref{ghat}, and \eqref{Le} are replaced subsequently by
\BEA
S^*(\epsilon,x,y)&:=&\frac{K(\epsilon,y,x)}{\int_{\mathcal{M}}K(\epsilon,x,y) \ dy} = \epsilon^{-d/2} \frac{K(\epsilon,y,x)} {G_\epsilon 1(x) },\nonumber \\
J^*_\epsilon u(x) &:=& \int_{\mathcal{M}}S^*(\epsilon,y,x) u(y) \ dy, \nonumber\\
L^*_\epsilon u(x) &:=& \frac{1}{\epsilon}(J_{\epsilon}^*- \mathcal{I}) u(x) = \mathcal{L}^*u(x) + \mathcal{O}(\epsilon).\label{Leadj}
\EEA

Note that the asterisk in the above are choices in notation and, while $G^*_{\epsilon}$ is the adjoint of $G_{\epsilon}$ with respect to $L^2(\mathcal{M})$, in general $J^*_\epsilon, L^*_\epsilon$ are not the adjoint of $J_{\epsilon}, \lep$, respectively.

The goal of manifold learning is to find a set of (basis) functions to describe the manifold from available data points $x_i$ that lie on the manifold $\mathcal{M}$. The local kernel theory introduced in \cite{bh:15vb} is a generalization of diffusion maps, a nonlinear manifold learning technique \cite{cl:06}. For the diffusion maps algorithm, the procedure above is carried out with $B(x)=0$ and $C(x) = \mathcal{I}(x)$. The result is a self-adjoint negative-definite Laplace-Beltrami operator $\mathcal{L}$ in an appropriate Hilbert space. In this case, the eigenfunctions of $\mathcal{L}$ form an orthonormal basis and the first few leading eigenfunctions, appropriately scaled to preserved the diffusion distance, are used as an isometric embedding to represent the manifold. In the case of a nonzero vector field $b$ and anisotropic diffusion tensor $c$, the resulting operator $\mathcal{L}$ is not self-adjoint. In \cite{bh:15vb},  eigenfunctions of the self-adjoint operator $\mathcal{L}+\mathcal{L}^*$ are considered for manifold learning. It should be noted that the evaluation of the prototypical kernel in \eqref{prototypicalkernel} requires the knowledge of either the intrinsic representation $b$ and $c$ together with the embedding function $\iota$ or the ambient representation $B$ and $C$ as shown in \eqref{Bb}-\eqref{Cc}.

In this paper, we approximate the solution of the boundary value problem in \eqref{PDE} using the integral operator in \eqref{Le}. That is, we approximate the boundary value problem in \eqref{PDE} as
\BEA
(a+\mathcal{L}) u \approx (a+L_\epsilon) u_\epsilon = f, \quad x\in \mathcal{M}. \label{approxlinearproblem}
\EEA
This approach is closely related to the point integral method (PIM) proposed in \cite{li2017point}. The relationship between the point integral method (PIM) and the local kernel method can be clarified through the following special example. 

Consider solving the Poisson problem with Neumann boundary conditions:
\BEA
\begin{cases} \label{poisson}
-\Delta u(x) &= f(x), \quad x\in \mathcal{M}\\
\frac{\partial u}{\partial n}(x) &= 0, \quad\quad x\in \partial\mathcal{M}.
\end{cases}
\EEA 
The point integral method uses a kernel of the form $$S(\epsilon,x,y) = C_\epsilon h\left(\frac{\|x-y\|^2}{4\epsilon}\right),$$ where $h:\mathbb{R}^+\to\mathbb{R}^+$ is either compactly supported or decaying exponentially (a local kernel) and $C_\epsilon$ denotes the normalization constant. Following the examples in \cite{li2017point}, we set $h(r) = e^{-r}$, so that $S$ is nothing but a Gaussian kernel. Define also $\tilde{S}(\epsilon,x,y) = C_\epsilon \tilde{h}(\frac{\|x-y\|^2}{4\epsilon})$ with $\tilde{h}(r) = \int_r^\infty h(s)ds$. For this example, it is clear that $\tilde{h}(r) =h(r)$ so $\tilde{S}(\epsilon,x,y)=S(\epsilon,x,y)$. For this special setup, the PIM approximates the solution of the Poisson problem in \eqref{poisson} with the solution of the following integral equation (see Eqn.~(1.2) in \cite{li2017point}):
\BEA
-\frac{1}{\epsilon} \int_\mathcal{M} S(\epsilon,x,y) (u(y) - u(x)) dy = \int_\mathcal{M}  \tilde{S}(\epsilon,x,y) f(y) dy   \nonumber
\EEA
which is nothing but:
\BEA
-L_\epsilon u(x) = \int_\mathcal{M}  S(\epsilon,x,y) f(y) dy = J_\epsilon f(x),\label{PIM}
\EEA
using the notations in \eqref{ghat} and \eqref{Le}. Since $\lim_{\epsilon\to 0}S(\epsilon,x,y) = \delta(\|x-y\|)$, in weak sense, it is clear that as $\epsilon \to 0$, the PIM described in \eqref{PIM} is equivalent to the local kernel approach, which approximates the Poisson problem in \eqref{poisson} with  $-L_\epsilon u(x) =f(x)$. 

While PIM can handle non-Neumann boundary conditions \cite{li2017point}, it is restricted to isotropic elliptic equations as noted in \cite{li2016convergent}. On the other hand, the local kernel theory is restricted to Neumann boundary conditions (if the manifold has boundary) but it can approximate general non-symmetric second-order linear elliptic differential operator since it uses the prototypical kernel in \eqref{prototypicalkernel}, which is a generalization of the PIM Gaussian kernel. Based on this observation, the local kernel approach can be interpreted as a generalization of PIM to non-symmetric second-order linear elliptic differential operators. This connection opens the door for possible generalization of the local kernel approach to other types of boundary conditions and/or an extension of PIM with prototypical kernel -- all of which are interesting future research directions.

Back to the local kernel approach, we now discuss the convergence of the approximation in \eqref{approxlinearproblem} under standard conditions for the existence of the weak solutions of \eqref{PDE}.

\comment{
We approximate the solution of $$\mathcal{L}u=f$$ with the solution of $\lep\hat{u}=f$, where $\lep$ is constructed via the prototypical kernel. If a unique solution to the latter problem exists, we can denote the solution operator as $\lep^{-1}$ and note that $\lep (u-\hat{u})=\lep u- f=\lep u-\ko u$ so 
\begin{align*}
\norm{u-\hat{u}}&=\norm{\lep^{-1}}\norm{\lep u-\ko u}.
\end{align*}

The preceding  gives us the consistency $\norm{\lep u-\ko u} \rightarrow 0$ and in this write up we show that $\norm{\lep^{-1}}$ is bounded and also that the discrete approximation of $\lep$ via Monte Carlo averages converges.}

\section{Approximate Linear Problem}

In this section we discuss the properties of the approximate linear problem in \eqref{approxlinearproblem}
and show how they relate to the solution of the linear problem in \eqref{PDE}. Specifically, we discuss the convergence of the approximate solution of \eqref{approxlinearproblem} to the exact solution \eqref{PDE} under the assumption that the latter exists and unique in Section 3.1. Subsequently, we discuss the minimum norm solution when the linear problem has non-unique solutions in Section~3.2. We close this section by discussing the discrete approximation of \eqref{approxlinearproblem}, describing the detailed of the algorithm for implementation, and the convergence rate.

\subsection{Convergence of the Approximate Solution}\label{section31}
As noted in the introduction, we consider the boundary value problem in \eqref{PDE} under the standard assumption that $\ko$ is uniformly elliptic with uniformly bounded coefficients. By the Fredholm alternative, a weak solution to \eqref{PDE} exists if and only if $f \in \mathcal{N}(\al)^{\perp}$. Note that here and in the remainder of the paper $\mathcal{N}(L), \mathcal{R}(L)$ denote the kernel and range of an operator $L$, respectively, and $A^{\perp}, \overline{A}$ the orthogonal complement and closure of $A$, respectively. Furthermore, this solution is unique if the homogenous problem corresponding to \eqref{PDE} has only the trivial solution $u \equiv 0$ (see \cite{bressan2013lecture,evans1998partial}). First, we show that a weak solution of the approximate problem in \eqref{approxlinearproblem} is also a weak solution of \eqref{PDE} up to order $\epsilon$. 
\begin{proposition}\label{prop1}
For every $\epsilon>0$, let $u_{\epsilon}$ be a weak solution of \eqref{approxlinearproblem}
with Neumann boundary conditions imposed in the case that $\mathcal{M}$ has boundary. Then $u_\epsilon$ is a weak solution of \eqref{PDE} up to order $\epsilon$.
\end{proposition}
\begin{proof}
We consider $a \equiv 0$. The proof for $a \not\equiv 0$ is entirely similar. Let $\phi\in C^{\infty}(\mathcal{M})$ and $u_{\epsilon}$ a weak solution to \eqref{approxlinearproblem}. Then 
\begin{align*}
\langle f,\phi \rangle &= \langle \lep u_{\epsilon},\phi\rangle \\
&=\frac{1}{\epsilon}\langle J_{\epsilon}u_{\epsilon},\phi\rangle-\frac{1}{\epsilon}\langle u_{\epsilon},\phi\rangle\\
&=\frac{1}{\epsilon}\langle G_{\epsilon}u_{\epsilon}, (G_{\epsilon}1)^{-1}\phi \rangle-\frac{1}{\epsilon}\langle u_{\epsilon},\phi\rangle\\
&=\frac{1}{\epsilon}\langle u_{\epsilon},G^*_{\epsilon}((G_{\epsilon}1)^{-1}\phi)\rangle-\frac{1}{\epsilon}\langle u_{\epsilon},\phi\rangle\\
&=\frac{1}{\epsilon}\langle u_{\epsilon},m\phi(G_{\epsilon}1)^{-1}+\epsilon(\omega\phi(G_{\epsilon}1)^{-1}+\al(m\phi 
(G_{\epsilon}1)^{-1}))\rangle
 -\frac{1}{\epsilon}\langle u_{\epsilon},\phi\rangle+\mathcal{O}(\epsilon)\\
&=\frac{1}{\epsilon}\langle u_{\epsilon},\phi (1-\epsilon\omega m^{-1})\rangle+\langle u_{\epsilon}, \omega \phi m^{-1}\rangle
 + \langle u_{\epsilon},\al\phi  \rangle-\frac{1}{\epsilon}\langle u_{\epsilon}, \phi \rangle+\mathcal{O}(\epsilon)\\
&=\langle u_{\epsilon},\al\phi\rangle+\mathcal{O}(\epsilon)\\
&=\langle \ko u_{\epsilon},\phi \rangle+\mathcal{O}(\epsilon).
\end{align*}
Here, we  used the expansions \eqref{aex} and $(G_{\epsilon}1)^{-1}=m^{-1}(1-\epsilon\omega m^{-1})+\mathcal{O}(\epsilon^2)=m^{-1}+\mathcal{O}(\epsilon)$. Thus, up to order $\epsilon$, $u_{\epsilon}$ is a weak solution of \eqref{PDE} and the proof is complete.
\end{proof}

\comment{
We develop a partial converse for the above statement. First, by the definition of the gradient and covarient derivative, 

\begin{align*}
\ko u&= b \cdot \nabla u+\frac{1}{2}\sum_{i,j} c_{ij} \nabla_i\nabla_j u\\
&=\sum_{i,j}g^{ij} \frac{\partial u}{\partial x^j}b_i +\frac{1}{2}\sum_{i,j}c_{ij}\nabla_i(\nabla u)_j\\
&=\sum_{i,j}g^{ij} \frac{\partial u}{\partial x^j}b_i +\frac{1}{2}\sum_{i,j}c_{ij}\left(\frac{\partial^2 u}{\partial x^i\partial x^j}-\sum_k \frac{\partial u}{\partial x^k}\Gamma^k_{j,i}\right)\\
&=\frac{1}{2}\sum_{i,j} c_{ij}\frac{\partial^2 u}{\partial x^i\partial x^j}+\sum_{i,j}g^{ij} \frac{\partial u}{\partial x^j}b_i-\frac{1}{2}\sum_k \frac{\partial u}{\partial x^k}\Gamma^k_{j,i}.
\end{align*}

Here $\Gamma_{j,i}^k$ is a Christoffel symbol of the second kind. We can rewrite the above in the following standard form:

\begin{align*}
\ko u&= \frac{1}{2}\sum_{i,j} c_{ij}\frac{\partial^2 u}{\partial x^i\partial x^j}+\sum_{i,j}g^{ij} \frac{\partial u}{\partial x^j}b_i-\frac{1}{2}\sum_k \frac{\partial u}{\partial x^k}\Gamma^k_{j,i}+a u\\
&=-\sum_{i,j} (-c_{ij}u_{x_i})_{x_j}+\left(\sum_i k_i+\sum_j(-c_{ij})_{x_j}\right)u_{x_i}+a u,
\end{align*}
where $k_i=b_i\sum_n g^{in}-\frac{1}{2}\sum_{l,k}\Gamma^i_{k,l}$. }

This proposition also shows that $L_\epsilon \to \mathcal{L}$ weakly as $\epsilon\to 0$. We can now prove the convergence of the approximate solution for well-posed linear problems:

\begin{theorem}\label{thm2}
\comment{Let $f \in \mathcal{N}(\al)^{\perp}$ } Assume $\ko$ is uniformly elliptic with uniformly bounded coefficients and that $a$ is defined such that $\mathcal{L}+a\mathcal{I}$ is strictly negative definite operator. Then the approximate problem \eqref{approxlinearproblem} also has a unique solution. For any $\epsilon>0$, if $u_\epsilon$ is the solution of \eqref{approxlinearproblem} and $u$ is the solution of \eqref{PDE}, then $u_\epsilon$ converges weakly to $u$.
\end{theorem}

\begin{proof}
The assumption guarantees the existence of unique weak solution of \eqref{PDE}. Since $\mathcal{L}$ is the generator of an ergodic It\^o diffusion, then $\mathcal{L}$ is negative definite with zero as its the largest eigenvalue. Since $\mathcal{L}+a\mathcal{I}$ is a strictly negative operator, then there exists $\alpha>0$ such that $\langle au, u\rangle \leq -\alpha \|u\|^2$, for any $u\in L^2(\mathcal{M})$. 

For any $\epsilon>0$, take any $u\in L^2(\mathcal{M})$ and note that
\begin{align*}
\langle (L_\epsilon + a\mathcal{I})u,u \rangle& =\frac{1}{\epsilon}\left(\langle J_\epsilon u, u\rangle - \norm{u}^2 \right) +\langle a u, u \rangle\\
&\leq \frac{1}{\epsilon}\norm{u}^2\left( \norm{J_{\epsilon}}-1 \right)+\langle au,u\rangle\\
&\leq-\alpha\norm{u}^2
\end{align*}
where we have used cauchy-schwarz and the fact that $J_\epsilon$ is a compact operator with $\|J_\epsilon\|\leq 1$. Thus, $L_\epsilon+a\mathcal{I}$ is strictly negative definite and by Theorem 5.12 in \cite{bressan2013lecture} the linear problem in \eqref{approxlinearproblem} has a unique solution and the inverse operator $(L_\epsilon+a\mathcal{I})^{-1}$ is bounded, in fact, $\|(L_\epsilon+a\mathcal{I})^{-1} \| \leq \alpha^{-1}$.

\comment{
Suppose $u_\epsilon$ is the weak solution of the approximate problem in \eqref{approxlinearproblem}, then by Proposition~\ref{prop1} it is also a weak solution of \eqref{PDE} up to order $\epsilon$. 
This means that for any $w\in C^\infty(\mathcal{M})$,
\BEA
\langle (\mathcal{L}+a\mathcal{I}) (u - u_\epsilon), w \rangle = \mathcal{O}(\epsilon),
\EEA 
where $u$ denotes the unique weak solution of \eqref{PDE}. Since smooth function is dense in $L^2(\mathcal{M})$, take $w =  (\mathcal{L}+a\mathcal{I}) (u - u_\epsilon)+\mathcal{O}(\epsilon)$, we have $\|(\mathcal{L}+a\mathcal{I}) (u - u_\epsilon) \| = \mathcal{O}(\epsilon^{1/2})$. This implies that,
\BEA
\|u-u_\epsilon \| \leq \|(\mathcal{L}+a\mathcal{I})^{-1}\| \|(\mathcal{L} +a\mathcal{I} ) (u - u_\epsilon) \| = \mathcal{O}(\epsilon^{1/2}),
\EEA
where we have used the fact that the solution operator, $(\mathcal{L}+a\mathcal{I})^{-1}$, is uniformly bounded, which follows directly by the assumption that $\mathcal{L}+a\mathcal{I}$ is strictly positive definite.
}

Finally, if $u_{\epsilon}$ is the solution of \eqref{approxlinearproblem} and $u$ the solution of \eqref{PDE}, then $$(\lep+a\mathcal{I}) (u-u_{\epsilon})=(\lep+a\mathcal{I}) u- f=(\lep+a\mathcal{I}) u-(\ko+a\mathcal{I}) u.$$ 
For $w\in C^\infty(\mathcal{M})$, we can now deduce that,
\BEA
\langle u-u_{\epsilon}, w \rangle  &=& \langle (\lep+a\mathcal{I})^{-1} (\lep u-\ko u),w\rangle  \nonumber \\ &=&  \langle  (\lep u-\ko u), (\lep+a\mathcal{I})^{-*} w\rangle = \mathcal{O}(\epsilon), \nonumber\EEA
where we have used the fact that adjoint of bounded linear operator exists and is also bounded and applied the Proposition~\ref{prop1} by noting that $(\lep+a\mathcal{I})^{-*} w \in C^\infty(\mathcal{M})$. 

\comment{\color{blue} Thus,
\BEA 
\|u-u_{\epsilon} \|_{L^2} &\leq& \|(L_\epsilon+a\mathcal{I})^{-1} \|_{L^2} \|L_\epsilon u - \mathcal{L} u \|_{L^2} \nonumber \\
&\leq& \alpha^{-1}vol({\mathcal{M}}) \|L_\epsilon u - \mathcal{L} u\|_{L^\infty} \leq  C\epsilon,
\EEA
for some constant $C = M\alpha^{-1}vol({\mathcal{M}})$ with $M$ independent of $\epsilon$. Here, we have used the fact that $L^\infty \subset L^2$ for compact manifold $\mathcal{M}$ and the pointwise convergence of $L_\epsilon$ to $\mathcal{L}$, noted in \eqref{Le}, }

\end{proof}

\comment{\color{blue}
Note in the above proof that if $\|1-a\epsilon \|>1$ then, using the fact that $\|J_\epsilon^*\|\leq 1$ along with the expansion \eqref{Leadj}, we see that
\begin{align*}
(\lep^*+a\mathcal{I})^{-1} &=\frac{\epsilon}{1-\epsilon a} \bigg(\frac{J_{\epsilon}^*}{1-\epsilon a}-\mathcal{I}\bigg)\\
&=-\frac{\epsilon}{1-\epsilon a}\bigg(\mathcal{I}+\frac{J_{\epsilon}^*}{1-\epsilon a}+\frac{J_{\epsilon}^{*2}}{(1-\epsilon a)^2}+\ldots\bigg)=\mathcal{O}(\epsilon).
\end{align*}
Note that in the second line above we have expanded the inverse as a Neumann series and in the last line we used that fact that $$\frac{\epsilon}{1-a\epsilon}=\epsilon(1+a\epsilon+\mathcal{O}(\epsilon^2)).$$ Thus the solution converges weakly at the rate $\mathcal{O}(\epsilon^2)$, which is faster than the $\mathcal{O}(\epsilon)$ convergence rate of the operator $L_{\epsilon}$ to $\mathcal{L}$. In Section~\ref{sec4}, we shall see from numerical experiments that the error in estimating $u$ is more accurate than that of $\mathcal{L}u$ in all examples, even for those for which $a=0$.
}

\subsection{Minimum norm solution}

In this section, we consider the case when $a = 0$ so that $L_{\epsilon}$ is not invertible. \comment{ In fact, in the previous section we saw that $\hat{G}_{\epsilon}$ is the semigroup operator associated with $\ko$. If the system is ergodic with equilibrium density $q$, then $\hat{G}_{\epsilon} \mathbbm{1}(x)=\mathbbm{1}(x)$ and $\hat{G}^*_{\epsilon} q=q$. In this case, $\norm{\hat{G}_{\epsilon}}\leq 1$ so that $\hat{G}_{\epsilon}-\mathcal{I}$ is not invertible. } In this case, the solutions of \eqref{PDE} and \eqref{approxlinearproblem} can instead be studied via the Fredholm alternative. 
\comment{Note that the Neumann boundary conditions are assumed when applicable. }
That is, due to the noninvertibility of $\lep$, we are in the second case of the Fredholm alternative: $\lep u=f$ is solvable if and only if$f \in \ker(\alep)^{\perp}$. In this case the solution is not unique and the numerical solution of the integral equation will in general not depend continuously on the data. Classically, integral equations of this type have been solved in one of the following ways: (1) by replacing the kernel so that the equation has an exact solution, (2) by using iterative methods (such as conjugate gradient descent),  (3) by restricting the solution to be of minimum norm \cite{lonseth1954}, or (4) by recasting the equation in a form that is uniquely solvable \cite{Atkinson1967}. In this section, we will find the unique minimum norm solution to \eqref{approxlinearproblem} using the generalized inverse $\lep^{\dagger}$ of $\lep$. Basically, if we impose the restriction that the solution must have minimum norm, then the numerical solution will depend continuously on the data. 

We first summarize (from \cite{KAMMERER1972547}) the relevant facts about the generalized inverse of a bounded linear operator. Let $X,Y$ be Hilbert spaces. Then any bounded linear operator $T:X\rightarrow Y$ decomposes $X$ and $Y$ as 
\BEA
X&=&\mathcal{N}(T)\oplus \mathcal{N}(T)^{\perp} \nonumber\\
Y&=&\mathcal{N}(T^*)\oplus \mathcal{N}(T^*)^{\perp}.\nonumber
\EEA

 Defining $P$ and $Q$ to be the orthogonal projection onto $\mathcal{N}(T)^{\perp}$, and $\mathcal{N}(T^*)^{\perp}$, respectively, we find that the problem $Tx=Qy$ has a solution for each $y \in Y$. The solution set is a convex subset of $X$ and contains a unique element of minimum norm. The generalized inverse $T^{\dagger}$ is defined as the linear operator that assigns to each $y \in Y$ the element of minimum norm among those that solve $Tx=Qy$. It may be the case that $T^{\dagger}$ is not bounded. However, for operators with closed range, the generalized inverse is bounded.  
In our case, the operator $\hat{G}_\epsilon - \mathcal{I}$, where $\hat{G}_\epsilon$ is compact, has closed range so we are guaranteed the existence of a bounded generalized inverse operator $\lep^{\dagger}$. With this background, we can establish the following result:

\begin{theorem}\label{thm3}
Assume $f \in \mathcal{N}(\al)^{\perp}$ and $\mathcal{L}$ is uniformly elliptic. Then $u_\epsilon: = L_\epsilon^\dagger f$ is a weak solution of \eqref{PDE} up to order $\epsilon$.  
\end{theorem}

\begin{proof}
Let $f_n \in \mathcal{R}(\mathcal{L}) $ be a sequence that converges to $f \in \overline{\mathcal{R}(\mathcal{L})} = \mathcal{N}(\al)^{\perp}$. Then, there exists $u_n\in H^1(\mathcal{M})$ such that $\forall w\in C^\infty(\mathcal{M})$,
\BEA
\langle \mathcal{L}u_n, w\rangle= \langle f_n, w\rangle =  \langle f, w\rangle + \mathcal{O}(\epsilon).\nonumber
\EEA
From proposition~\ref{prop1}, $\mathcal{L}u_n = L_\epsilon u_n + \mathcal{O}(\epsilon)$ weakly, so,
\BEA
\langle L_\epsilon u_n, w\rangle= \langle f, w\rangle + \mathcal{O}(\epsilon),\nonumber
\EEA
which means that  $f\in \mathcal{R}(L_\epsilon)$ up to order $\epsilon$. This implies that if $u_\epsilon =  L_\epsilon^\dagger f$, then 
\BEA
\langle L_\epsilon u_\epsilon, w\rangle = \langle L_\epsilon L_\epsilon^\dagger f, w \rangle  = \langle f, w\rangle + \mathcal{O}(\epsilon).
\EEA
We have shown that $u_\epsilon$ is a weak solution of the approximate linear problem in \eqref{approxlinearproblem} up to order $\epsilon$. By Proposition~\ref{prop1}, $u_\epsilon$ is a weak solution of \eqref{PDE} up to order $\epsilon$. 
\end{proof}

\subsection{Discrete Approximation}\label{section33}

Here we outline the procedure to construct a numerical approximation $\hat{L}_{\epsilon}$ of $\lep$ in the case when the nodes, drift and diffusion tems are given either in the intrinsic coordinates of $\mathcal{M}$ with known embedding $\iota$ or in the ambient space with unknown embedding. The only difference between the two cases is that if the embedding is known then we can represent the solution in the intrinsic coordinates rather than in $\R^n$. However, in both cases, the kernel is formed by evaluating the nodes in $\R^n$. Note that we discretize continuous functions and tensors by  representing them in the delta basis of the nodes and approximate integrals using Monte-Carlo averages.

We first describe the approximation procedure in the first case under the simple setting when the $d$-dimensional manifold $\mathcal{M}$ is embedded in $\R^n$ and $(x_i)_{i=1}^N \in \R^n$ are uniformly spaced nodes. We discretize $b(x)$ and $c(x)$ as  $b(x_i)$ and $c(x_i)$ for each $i$ and lift these into $\R^n$ using \eqref{Bb}, \eqref{Cc}.  We can now use the usual Monte-Carlo approach to approximate $L_{\epsilon}$:

\begin{enumerate}[1]
\item Discretize $K(\epsilon,x,y)$ as an $N \times N$ matrix $\hat{K_\epsilon}$ where $(\hat{K_{\epsilon}})_{ij}=K(\epsilon,x_i,x_j)$.
\item  For each $i$, approximate $G_\epsilon1(x_i)=\int_\mathcal{M} K(\epsilon,x_i,y) \ dy$ as \\   $(\hat{G_{\epsilon}}\vec{1})_i:= \frac{1}{N}\sum_{j=1}^N(\hat{K_{\epsilon}})_{ij},$ where $\vec{1}$ denote an $N$-dimensional vector with 1 as its components. Create an $N \times N$ diagonal matrix $D$ with $D_{ii}=(\hat{G_{\epsilon}}\vec{1})_i$.
\item Discretize $\hat S(\epsilon,x,y)$ as an $N \times N$ matrix $\hat{S_{\epsilon}}$ where $$\hat{S_\epsilon}:= D^{-1} \hat K_{\epsilon}.$$
\item Finally, $\lep$ is approximated by the $N\times N$ matrix $$\hat \lep:=\frac{1}{\epsilon}\left( \hat S_{\epsilon}-\mathcal{I}_N\right)=\frac{1}{\epsilon}\left(D^{-1}\hat K_\epsilon-\mathcal{I}_N\right),$$ 
where $\mathcal{I}_N$ denotes an identity matrix of size $N\times N$.
\end{enumerate}
Since the kernel is exponentially decaying, it is usual practice to use a $k$-nearest neighbors algorithm to introduce sparsity into the matrix approximation of $\hat K$. 

Following \cite{Harlim}, we can tune $\epsilon$ by defining $Q(\epsilon)=\frac{1}{N^2}\sum_{i,j}K(\epsilon,x_i,x_j)$ and searching for the region where $\log(Q(\epsilon))$ grows linearly. Empirical results suggest that the dimension $d$ of $\mathcal{M}$ can be approximated by $$d=2\times \mathrm{max}\left\{\frac{d\log(Q(\epsilon))}{d\log(\epsilon)}\right\}$$ and we set $\epsilon$ to be the value corresponding the estimated $d$. While this automated tuning strategy may not necessarily give the best estimates on the resulting operator estimation, it is convenient and numerically cheap. For a theoretically justified yet computationally more elaborate technique, one can also use the local singular value decomposition technique deduced in \cite{IDM}. In our numerical experiments below, we will only use either the empirical tuning mentioned above or a simple empirical tuning by comparing the estimates to the true solution when the latter is known.

Next, we relax the uniform spacing assumption and assume that $(x_i)_{i=1}^N \in \R^n$ are sampled independently from a density $q$. In this case, every integral (or discrete Monte-Carlo approximation) is with respect to the sampling density $q$ and therefore we need to modify the algorithm to debias this sampling effect.
To do this, define the exponential kernel $$h(\tilde\epsilon,x,y):= \exp \left(-\frac{\|x-y\|^2}{2\tilde\epsilon}\right),$$ and let $H_{\tilde\epsilon} u(x)=\te^{-d/2}\int_\mathcal{M} h(\tilde\epsilon,x,y)u(y) \ dy$ be the corresponding Fredholm operator.  Subsequently define 
\BEA \nonumber q_{\tilde\epsilon} (x)&:=&H_{\te}q(x)\\  \nonumber \gq u(x)&:= &G_\epsilon(q(x)u(x)) \\ \nonumber  \tilde{q}_{\epsilon,\te}(x)&:=&\gq\left(\qt(x)^{-1}\right)\\
L_{\epsilon,\te} u(x)&:=&\frac{1}{\epsilon}\left(\tilde{q}_{\epsilon,\te}^{-1}(x)\gq\left( u(x) q_{\te}^{-1}(x)\right)-u(x)\right) \label{Lee}
\EEA 

The next proposition shows how to debias the Kolmogorov operator.

\begin{proposition}\label{prop4}
\label{bias} For any $u \in C^3(\mathcal{M})$, where $\mathcal{M}$ denotes a $d$-dimensional manifold embedded in $\mathbb{R}^n$,
$$L_{\epsilon,\te} u(x)=\ko u(x)+\mathcal{O}(\te,\epsilon),$$ for each $x\in\mathcal{M}$.
\end{proposition}

\begin{proof}
From \cite{cl:06}, we have
\[H_{\te} u(x):= \te^{-d/2} \int_{\mathcal{M}} h(\te,x,y) u(y)\, dy = \tmo u(x)+\te\tmi(\tom(x)u(x)+\Delta u(x))+\mathcal{O}(\te^2),\] 
where $\tilde{m}_0 := \int_{\mathbb{R}^d} h(\|z\|^2)\,dz$, $\tilde{m}_1 := \frac{1}{2} \int_{\mathbb{R}^d} z_1^2 h(\|z\|^2)\,dz$ and $\omega$ depends on the induced geometry of $\mathcal{M}$.
Consequently 
$$\qt^{\alpha}=\tmo^{\alpha} q^{\alpha}\left(1+\te \tm\tom q+\te\tm q^{-1}\Delta q\right)^{\alpha}+\mt,$$
 where $\tm=\tmi/\tmo$. Note that we have introduced the real parameter $\alpha$ for computational convenience in proving the next lemma as shown in the appendix.

From \eqref{gef}, we can deduce
\BEA
\gq\left(u q_{\te}^{-\alpha}\right)&=& mu q \tilde{q}_{\epsilon,\te}^{-\alpha}\left(1+\epsilon m^{-1}\omega +\epsilon (uq)^{-1}\tilde{q}_{\epsilon,\te}^{\alpha}\ko(uq\tilde{q}_{\epsilon,\te}^{-\alpha})\right)+\me \nonumber\\
&=& m\tmo^{-\alpha} uq^{1-\alpha}\left(1-\alpha\te \tm\tom -\alpha\te\tm q^{-1}\Delta q+\mt\right) \nonumber\\
 &&\hspace{.5 in} \times \left(1+\epsilon m^{-1}\omega +\epsilon\frac{\ko(uq^{1-\alpha})}{uq^{1-\alpha}}+\mathcal{O}(\epsilon\te,\epsilon^2)\right)\nonumber\\
&=& m\tmo^{-\alpha} uq^{1-\alpha}  \left(1-\alpha\te \tm\tom -\alpha\te\tm q^{-1}\Delta q+\epsilon m^{-1}\omega +\epsilon\frac{\ko(uq^{1-\alpha})}{uq^{1-\alpha}} \right)\nonumber \\
&&\hspace{.5 in} + \mathcal{O}(\epsilon\te,\epsilon^2,\te^2).\label{Gqe}
\EEA

\noindent Setting $u(x)=1(x)$ and $\alpha=1$, we obtain 
\BEA 
\nonumber
\tilde{q}_{\epsilon,\te} &= m\tmo^{-1}  \left(1-\te \tm\tom -\te\tm q^{-1}\Delta q+\epsilon m^{-1}\omega \right)  + \mathcal{O}(\epsilon\te,\epsilon^2,\te^2). 
\EEA
Therefore, setting $\alpha = 1$, we have,
\begin{align*}
\tilde{q}_{\epsilon,\te}^{-1}\gq\left(uq_{\te}^{-1}\right)&=u \left(1-\te \tm\tom -\te\tm q^{-1}\Delta q+\epsilon m^{-1}\omega +\epsilon\frac{\ko(u)}{u} \right)\\
&\hspace{.5 in} \times \left(1+\te \tm\tom +\te\tm q^{-1}\Delta q-\epsilon m^{-1}\omega \right)  + \mathcal{O}(\epsilon\te,\epsilon^2,\te^2)\\
&= u  + \epsilon \ko u + \mathcal{O}(\epsilon\te,\epsilon^2,\te^2).
\end{align*}
By \eqref{Lee}, the result follows immediately.

\end{proof}

To show that Proposition~\ref{prop1} still holds when $L_\epsilon$ in \eqref{approxlinearproblem} is replaced by $L_{\epsilon,\te}$ as defined in \eqref{Lee}, it is sufficient to note that for any 
 smooth $\phi$,
\begin{align*}
qq_{\te}^{-1}G_{\epsilon}^*(\tilde{q}^{-1}_{\epsilon,\te}\phi)&=qq_{\te}^{-1}\bigg(m\tilde{q}_{\epsilon,\te}^{-1}\phi+\epsilon\bigg(\omega \tilde{q}_{\epsilon,\te}^{-1}\phi+\ko^*(m\tilde{q}_{\epsilon,\te}^{-1}\phi)\bigg)\bigg)+\mathcal{O}(\epsilon^2)\\
&=\tmo^{-1}\bigg(\tmo(1-\epsilon m^{-1}\omega)\phi+\epsilon\bigg(\omega m^{-1}\tmo\phi+\ko^*(\tmo\phi)\bigg)+\mathcal{O}(\epsilon,\te)\\
&=\phi+\epsilon\ko^*\phi+\mathcal{O}(\epsilon,\te).
\end{align*}
Consequently, Theorems~\ref{thm2} and \ref{thm3} hold when $L_{\epsilon}$ is replaced by $L_{\epsilon,\te}$.

\comment{
To show that this solution still converges when the operator is biased, we show that proposition \ref{prop1} still holds for the biased operator $L_{\epsilon,\te}$. To this end, it is enough to note that for any smooth $\phi$,
\begin{align*}
qq_{\te}^{-1}G_{\epsilon}^*(\tilde{q}^{-1}_{\epsilon,\te}\phi)&=qq_{\te}^{-1}\bigg(m\tilde{q}_{\epsilon,\te}^{-1}\phi+\epsilon\bigg(\omega \tilde{q}_{\epsilon,\te}^{-1}\phi+\ko^*(m\tilde{q}_{\epsilon,\te}^{-1}\phi)\bigg)\bigg)+\mathcal{O}(\epsilon^2)\\
&=\tmo^{-1}\bigg(\tmo(1-\epsilon m^{-1}\omega)\phi+\epsilon\bigg(\omega m^{-1}\tmo\phi+\ko^*(\tmo\phi)\bigg)+\mathcal{O}(\epsilon,\te)\\
&=\phi+\epsilon\ko^*\phi+\mathcal{O}(\epsilon,\te).
\end{align*}

Consequently, Proposition~\ref{prop1} still holds. That is, if $u_{\epsilon}$ is a weak solution to $$(a+L_{\epsilon,\te})u_{\epsilon}=f$$ then $u_{\epsilon}$ is a weak solution to \eqref{PDE} up to $\mathcal{O}(\epsilon,\te)$. This implies that Theorem~\ref{thm2} and \ref{thm3} hold for the solution to the debiased operator.}

The previous proposition justifies the use of the following discretization scheme to debias the Kolmogorov operator from non-uniformly distributed samples $\{x_i\}$:

\begin{enumerate}[1]
\item Discretize $h_{\tilde\epsilon}(x,y)$  as the $N\times N$ matrix $\hat h_{\epsilon}$, respectively.
\item  For each $i$, approximate $q_{\te}(x_i)$ with $(\hat H_{\tilde\epsilon} \vec{1})_i=\sum_{j=1}^N(\hat H_{\tilde\epsilon})_{ij}$. Create an $N \times N$ diagonal matrix $D_1$  with $(D_1)_{ii}=(\hat{H}_{\tilde\epsilon}\vec{1})_i$. 
\item (Debiasing Step) Discretize the kernel of the integral operator $G_{q,\epsilon}(uq_{\te}^{-1}$) by right normalizing $\hat{K}$. That is, approximate $K(\epsilon,x,y)q_{\te}^{-1}(y)$ by setting $\hat K \leftarrow \hat K D_1^{-1}$.  
\item Proceed with the left normalization and the formation of $L_{\epsilon,\te}$ as in steps $2-4$ in the uniform case.  In the rest of this section, we denote the discrete estimate as $\hat{L}_{\epsilon,\te}$.

\end{enumerate}

Finally, an approximate solution to \eqref{PDE} is given by a solution $\hat u$ of $(a+\hat L_{\epsilon,\te}) \hat{u} =\vec{f}$.
For appropriate $a$ and $\epsilon$, $a\mathcal{I}+L_{\epsilon,\te}$ is invertible as discussed in Section~\ref{section31}.  In this case, we will show that the discrete approximation constructed using the preceding algorithm is also invertible and that $\hat{u}=(a\mathcal{I}+\hat{L}_{\epsilon,\te})^{-1}\vec{f}$ approximates the solution $u$. In the case that $a\mathcal{I}+\hat{L}_{\epsilon,\te}$ is not invertible, we can form the pseudoinverse $(a\mathcal{I}+\hat{L}_{\epsilon,\te})^{\dagger}$ to approximate the minimum norm solution of \eqref{approxlinearproblem}.  

The convergence rate of the approximate solution depends on the following consistency result:

\begin{lemma}\label{lemma5}
Let $x_i\in\mathcal{M}\subseteq \mathbb{R}^n$ for $i=1,\ldots, N$ be i.i.d.~samples with sampling density $q(x)\in C^{3}(\mathcal{M})$ defined with respect to the volume form inherited by the $d$-dimensional manifold $\mathcal{M}$ from the ambient space $\mathbb{R}^n$. For any $u\in C^3(\mathcal{M})$, 
\[ | \hat{L}_{\epsilon,\te} u(x_i)-\mathcal{L} u(x_i) |= \mathcal{O}\left(\epsilon,\te,\frac{q(x_i)^{1/2}}{\sqrt{N}\te^{2+d/4}},\frac{\|\nabla_{\tilde{g}} u(x_i) \| q(x_i)^{-1/2 }}{\sqrt{N}\epsilon^{1/2+d/4}}\right),\]
in probability. Here, $L_{\epsilon,\te}$ is defined in \eqref{Lee}, the gradient operator is defined with respect to a new metric, $\tilde{g}(u,v) := g(c^{-1/2}u,c^{-1/2}v)$ for all $u,v \in T_{x}\mathcal{M}$, where $g$ denotes the Riemannian metric inherited by $\mathcal{M}$ from the ambient space and $c$ denotes the symmetric positive definite diffusion tensor.
\end{lemma}

The proof of this error bound follows closely the technique in \cite{singer2006graph,bh:15vb} and is given in the Appendix. We should point out that the notation $\|\cdot \|$ in this error bound denotes the norm with respect to the Riemannian metric $\tilde{g}$. The error term $\mathcal{O}(\epsilon,\te)$ describes the error of the continuous operator, $L_{\epsilon,\te}$, established in Proposition~\ref{prop4}. The third term, $\mathcal{O}\big(\frac{q(x_i)^{1/2}}{\sqrt{N}\te^{2+d/4}}\big)$ is the sampling error for obtaining an order-$\te^2$ estimate of $q_{\te}(x_i) = \te^{-d/2}\int_{\mathcal{M}} h(\te,x_i,y) q(y)dy$. The last error term, $\mathcal{O}(\frac{\|\nabla_{\tilde{g}} u(x_i) \| q(x_i)^{-1/2 }}{\sqrt{N}\epsilon^{1/2+d/4}})$, describes the error of approximating $L_{\epsilon,\te} u(x_i)$ with $\hat{L}_{\epsilon,\te} u(x_i)$. Notice that this error is large when $q$ is small.

From the error bound in Lemma~\ref{lemma5}, we can deduce:
\begin{theorem}\label{thm6}
Let $x_i\in\mathcal{M}\subseteq \mathbb{R}^n$ for $i=1,\ldots, N$ be i.i.d.~samples with positive sampling density $q\in C^{3}(\mathcal{M})$ defined with respect to the volume form inherited by the $d$-dimensional manifold $\mathcal{M}$ from the ambient space $\mathbb{R}^n$. For any $u\in C^3(\mathcal{M})$ that is the solution of \eqref{PDE} with strictly negative definite $a\mathcal{I}+\mathcal{L}$, the uniform error in estimating $u$ with $\hat{u}:=(a\mathcal{I}+\hat{L}_{\epsilon,\te})^{-1}\vec{f}$ is
\BEA
\|u-\hat{u} \|_\infty =\mathcal{O} \left(\epsilon,\te,\frac{ 1}{\sqrt{N}\te^{2+d/4}},\frac{1}{\sqrt{N}\epsilon^{1/2+d/4}}\right)\label{discreteerror}
\EEA
in probability. Here, the uniform norm is defined over $\mathbb{R}^N$.
\end{theorem}

\begin{proof}
We establish the stability of $(\hat{L}_{\epsilon,\te} +a\mathcal{I})$ and then use the consistency proved in the previous lemma to derive the uniform convergence rate of $\hat{u}$ to $u$.

In the proof of Theorem~\ref{thm2}, we showed that $a\mathcal{I}$ is a strictly negative definite operator. It is clear that the discretization of this operator on $x_i$ is a diagonal matrix with components, $a_i := a(x_i)<0$. Note that 
\BEA
\hat{L}_{\epsilon,\te} +a\mathcal{I} = \epsilon^{-1} (D^{-1}\hat{K}-(1-\epsilon a)\mathcal{I}), \nonumber
\EEA
where all components of $J:=D^{-1}\hat{K}$ are non-negative and $\sum_{j=1}^N J_{ij} = 1$, for all $i=1,\ldots,N$ (see Section~\ref{section33}). Define $A: = J-(1-\epsilon a)\mathcal{I}$ so that $\hat{L}_{\epsilon,\te} +a\mathcal{I}=\epsilon^{-1} A$. Since $0< J_{ii} \leq 1$ and $a_i<0$, it is clear that for any $\epsilon>0$,
\BEA
|A_{ii}| = |J_{ii}-1+\epsilon a_i| = -(J_{ii}-1+\epsilon a_i) > 1-J_{ii} = \sum_{j\neq i} J_{ij} = \sum_{j\neq i} |A_{ij}|.\nonumber
\EEA
for all $i$.  Thus $A$ is strictly diagonally dominant and consequently nonsingular. Using the Ahlberg-Nilson-Varah bound \cite{ahlberg1963,varah1975}, we obtain
\BEA
\|A^{-1}\|_{\infty} \leq \frac{1}{min_i (|A_{ii}|-\sum_{j\neq i} |A_{ij}|)} = \frac{1}{\epsilon\min_i (-a_i)} = \frac{1}{\epsilon \alpha},\nonumber
\EEA
where $\alpha := \min_i (-a_i)>0$.  Thus $\| (\hat{L}_{\epsilon,\te}+a\mathcal{I})^{-1}\|_{\infty} = \epsilon \|A^{-1}\|_{\infty} \leq \alpha^{-1}$. That is, the
matrix $\hat{L}_{\epsilon,\te}+a\mathcal{I}$ is nonsingular and its inverse is bounded uniformly, independent of $\epsilon$ and $N$. This establishes the stability of $\hat{L}_{\epsilon,\te} +a\mathcal{I}$.

Since
\BEA
(\hat{L}_{\epsilon,\te} +a\mathcal{I}) (u(x_i)-\hat{u}(x_i))&=&(\hat{L}_{\epsilon,\te}+a\mathcal{I}) u(x_i)- f (x_i)\nonumber\\ &=& (\hat{L}_{\epsilon,\te}+a\mathcal{I}) u(x_i)-(\ko+a\mathcal{I}) u(x_i),\nonumber\\ &=&
(\hat{L}_{\epsilon,\te}-\ko) u(x_i) \nonumber
\EEA
we can deduce that, as $\epsilon,\te\to 0$,
\BEA
|u(x_i) -\hat{u}(x_i)| &\leq &\| (\hat{L}_{\epsilon,\te} +a\mathcal{I})^{-1} \|_\infty \| \hat{L}_{\epsilon,\te} u-\mathcal{L} u \|_\infty \nonumber \\
&\leq& C \left(\epsilon,\te,\frac{ 1}{\sqrt{N}\te^{2+d/4}},\frac{1 }{\sqrt{N}\epsilon^{1/2+d/4}}\right),\nonumber
\EEA
where $C := K\alpha^{-1} \max\{q_{max}^{1/2},1,\max_i\|\nabla_{\tilde{g}} u(x_i) \| q_{min}^{-1/2}\} $ for some constant $K>0$ that is independent of $\epsilon$ and $N$. Since $\mathcal{M}$ is compact, it is clear that $\|\nabla_{\tilde{g}} u\|<\infty$ and $0<q_{min}\leq q\leq q_{max} < \infty$ and thus $C>0$ is finite and the proof is complete.
\end{proof}

This bound, however, is not sharp. For example, consider uniformly distributed grid points so that $q=1$ and the third error term can be neglected. Balancing the first and last error terms, we obtain $\epsilon = \hat{C} N^{-1/(3+d/2)}$, for some constant $\hat{C}$ that depends on the geometry of the manifold, as pointed out in \cite{singer2006graph}. In one of the examples below (see Section~4,1), we numerically found that the convergence rate is much faster with rate $N^{-2}$ for a $d=1$ dimension problem.

\section{Numerical examples}\label{sec4}

In this section, we demonstrate the numerical performance of the local kernel method on various test examples. We begin with a simple example involving a linear differential equation on a flat domain $[0,1]$. Subsequently we show numerical results involving variable coefficient differential equations on non-isometrically embedded smooth manifolds, such as full and half ellipses in $\mathbb{R}^2$ and full and half three-dimensional tori. Finally, we will show an example with unknown embedding where the functions and data are given in the ambient coordinates. In this section, we will ignore the subscript $\{\epsilon,\te\}$ for the discrete estimate $\hat{L}$ for notational simplicity.

\subsection{Linear differential equations on $[0,1]$}

In this first example, we consider solving a linear Boundary Value Problem (BVP),
\BEA
\begin{aligned}\label{bvp}
(\mathcal{L}-2\mathcal{I})u(x) &:=\frac{1}{2}cu''(x) + b u'(x) +au(x) = f(x), \quad x\in (0,1)\\
u'(0) &= u'(1) = 0,
\end{aligned}
\EEA
where $a =2, b=2, c=1$. In this simple example, one can verify that for 
\BEA
f(x) = -4\pi \sin(2\pi x)- (2\pi^2+2) \cos(2\pi x),
\EEA
 the analytical solution for the BVP in \eqref{bvp} is $u(x)= \cos(2\pi x)$. For this problem, one can verify that 
$\mathcal{L}-2\mathcal{I}$ is invertible so the existence of the weak solution is guaranteed using the standard Lax-Milgram argument.

In our numerical experiment, we apply the prototypical kernel in \eqref{prototypicalkernel} with $B = b =2, C = c = 1$ and $\epsilon=2\times 10^{-6}$. For this example, since the protopytical kernel is simply a Gaussian kernel with uniform covariance, we will also use this kernel as $h_{\te}$ for the right-normalization. Thus, $\te = \epsilon = 2\times 10^{-6}$. Under these specifications, we construct an $N\times N$ matrix $\hat L_\epsilon$ on $N=1000$ equally spaced discrete points $\{x_i=i/N \}_{i=0,\ldots,N}$ on $[0,1]$. For efficient computation, a sparse matrix representation of the prototypical kernel is generated by only evaluating it on $k=100$ nearest neighbors (based on the usual Euclidean vector distance) of each $x_i$.

In the remainder of this section, we will use $\vec{u}$ and $\vec{f}$ to denote $N$- dimensional vectors whose $i$th components are $u(x_i)$ and $f(x_i)$, respectively. In the top panel of Figure~\ref{odeexamplefig1}, we compare  $(\hat L-2\mathcal{I}_N) \vec{u}$ with the analytic $\vec{f}$. The error in the operator estimation is $\|(\hat L-2\mathcal{I})\vec{u} - \vec{f} \|_\infty = 4.3870$.  This large error occurs at the boundaries as expected since the asymptotic expansion in \eqref{gef} only holds away from the boundary. Away from the boundaries, the differences between $(\hat{L}-2\mathcal{I}) \vec{u}$ and $\vec{f}$ are on the order of $10^{-4}-10^{-3}$. In the bottom panel of Figure~\ref{odeexamplefig1}, we compare the discrete estimate $\hat{u} = (\hat L-2\mathcal{I})^{-1}\vec{f}$ and the analytical solution $u(x_i)$. In this case, the error of the approximate solution is $\|\vec{u}-\hat{u}\|_\infty = 0.0019$.

\begin{figure}
\centering
\includegraphics[width = .7\textwidth]{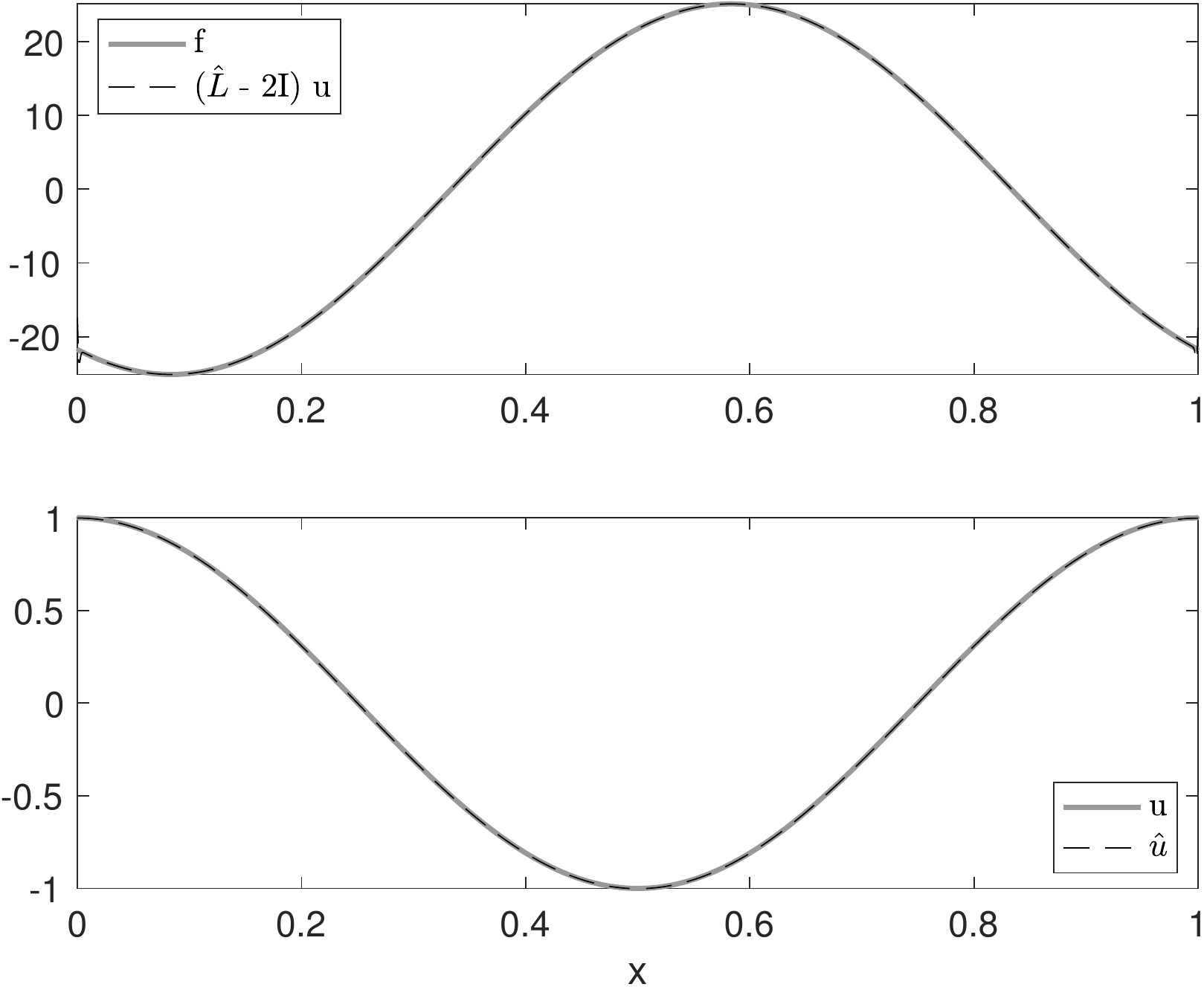}
\caption{Linear boundary value problem in \eqref{bvp}: Pointwise operator estimation (top) and approximate solution by direct inversion (bottom).}
\label{odeexamplefig1}
\end{figure}

In Figure~\ref{ratio}, we show the uniform error as a function of the ratio of the advection, $b$, and diffusion, $c$, terms in \eqref{bvp}, with fixed diffusion coefficient $c=1$. Notice that the error grows as the advection becomes dominant. In particular, the error increases significantly from $10^{-2}$ as the ratio $b/c > 100$. In Figure~\ref{errorrate}, we show the convergence rate in terms of the number of uniformly distributed grid points. Notice that the error rate is close to $N^{-2}$, which is much faster than the estimate in Theorem~\ref{thm6}. In the same figure, we also show the value of  the bandwidth parameter $\epsilon$ that is used in the local kernel, which is of order $N^2$ as well. This parameter is empirically chosen to minimize the error $\|\hat{u}-u\|_{\infty}$ for each $N$, which is possible in this example since the true solution $u$ is known. This numerical result also demonstrates that the error in estimating $u$ is on the same order as the bandwidth parameter value $\epsilon$ (see the right panel of Figure~\ref{errorrate}).

\begin{figure}
\centering
\includegraphics[width = .6\textwidth]{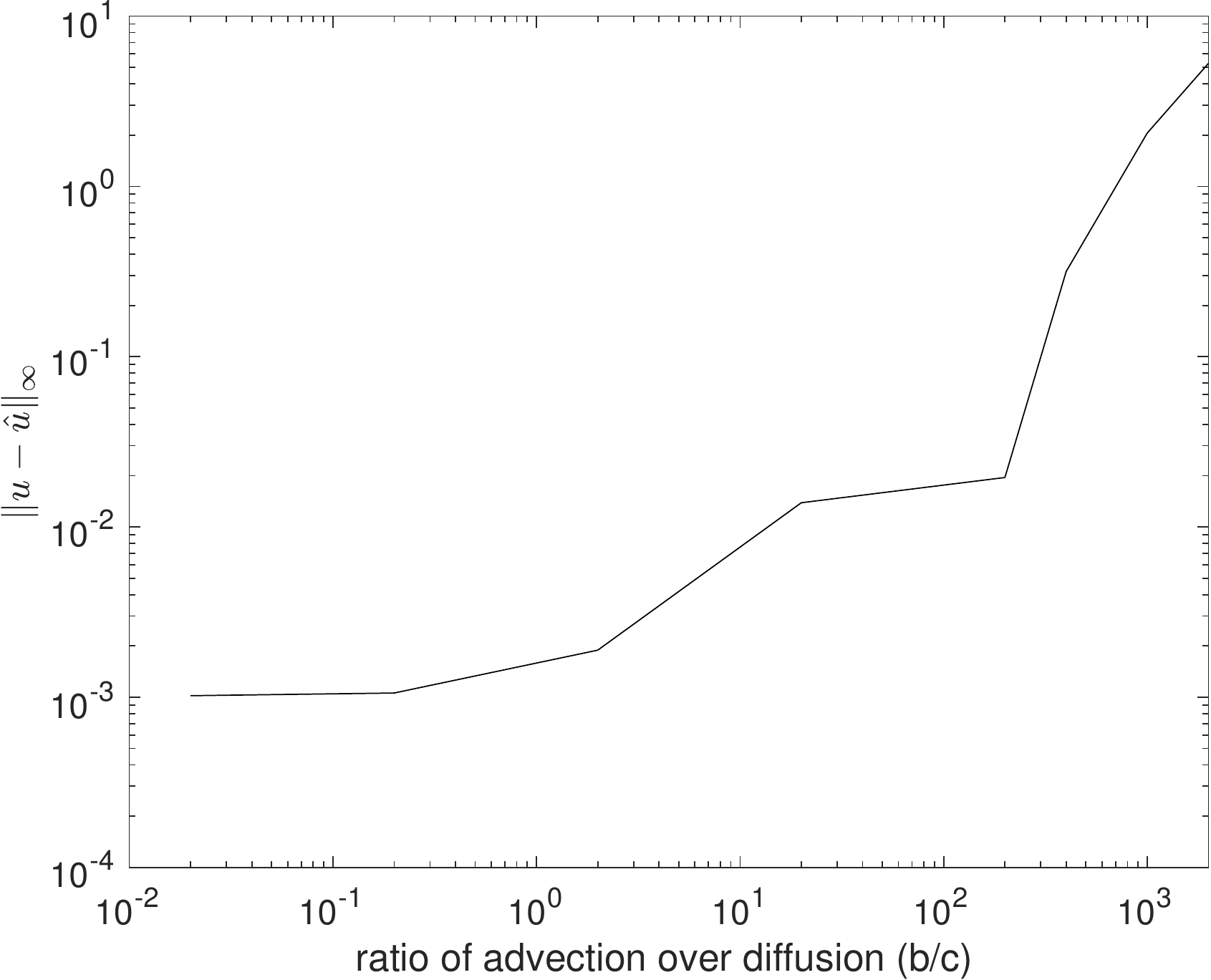}
\caption{Linear boundary value problem in \eqref{bvp}: The uniform errors as functions of the ratio between the advection, $b$, and the diffusion, $c$, coefficients.}
\label{ratio}
\end{figure}

\begin{figure}
\centering
\includegraphics[width = .48\textwidth]{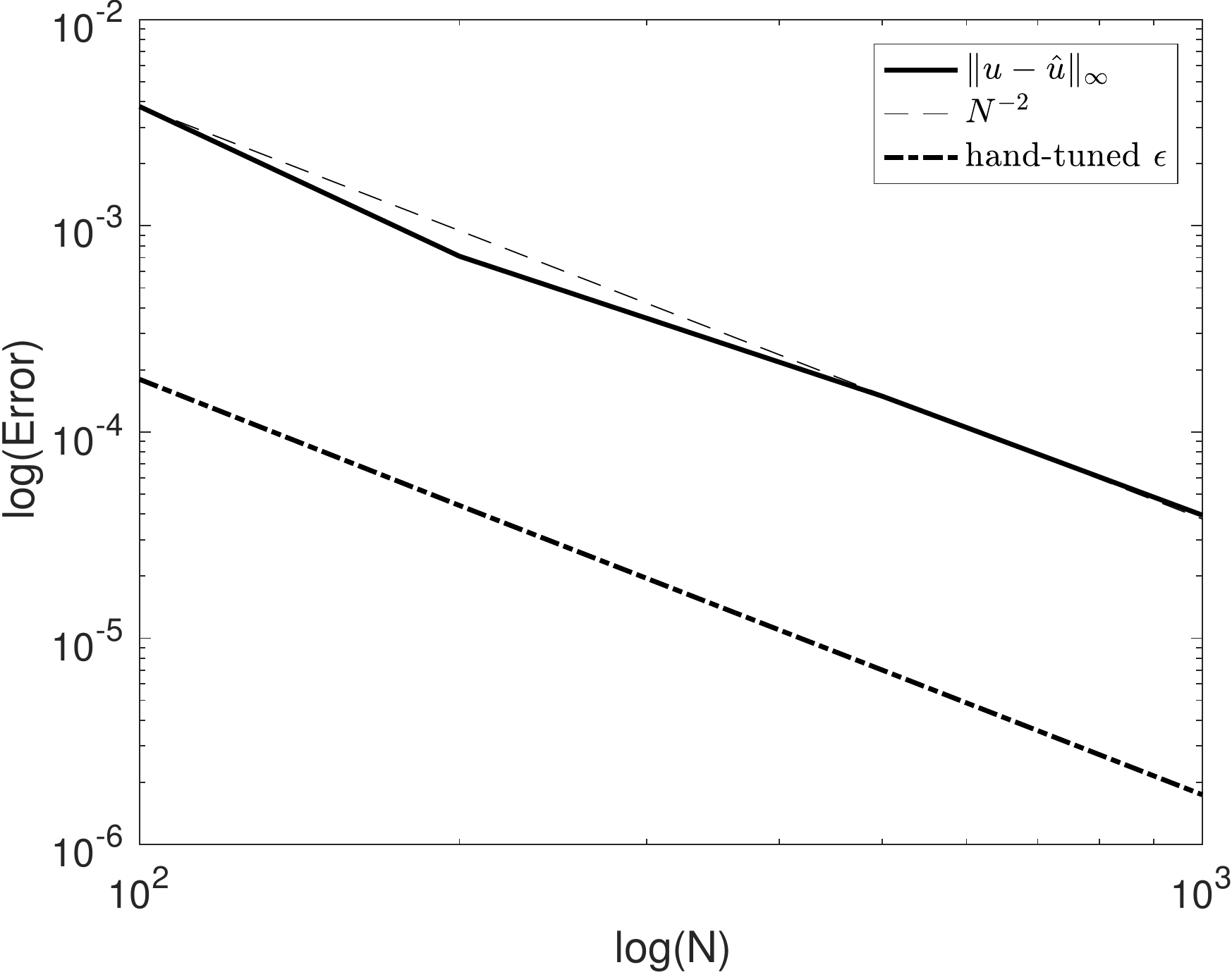}
\includegraphics[width = .48\textwidth]{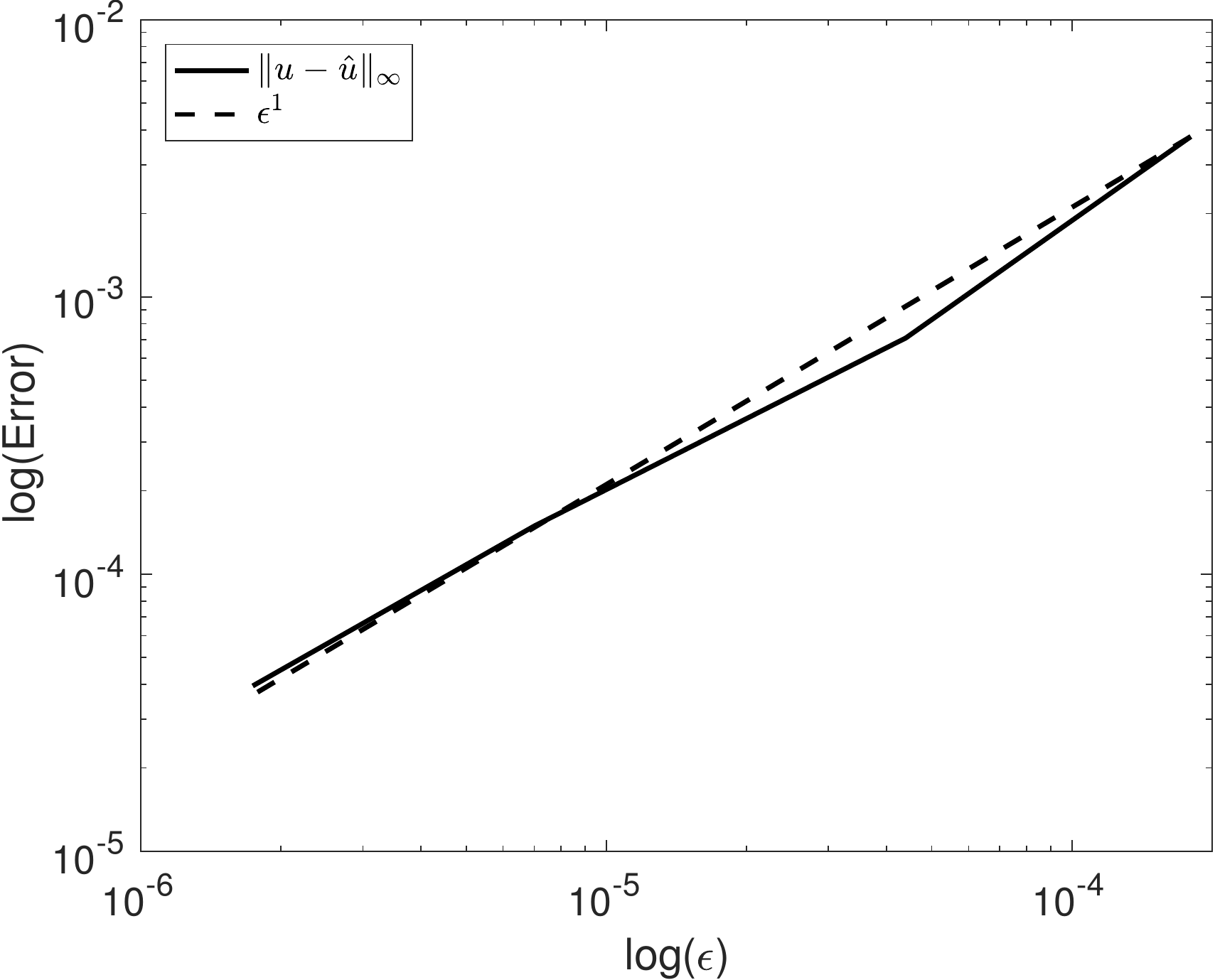}
\caption{Linear boundary value problem in \eqref{bvp}: Uniform error as a function of the total number of grid points, $N$ (left panel). We also show the value of $\epsilon$ that is used in the local kernel. This parameter is empirically tuned to minimize the error for each $N$. Uniform error as a function of $\epsilon$ is shown in the right panel.}
\label{errorrate}
\end{figure}

\subsection{Variable coefficients differential equation on full and partial ellipses}\label{1dexamples}

In the second example, we consider solving the boundary value problem in \eqref{PDE} with $a=0$ on an ellipse $\mathcal{M}\subset \mathbb{R}^2$ where the differential operator $\mathcal{L}$ is defined as in \eqref{generalL} with:
\BEA
\begin{aligned}
b(\theta) &= \cos\theta, \\
c(\theta) &= 1.1 +\cos\theta .
\end{aligned}
\EEA
For this numerical demonstration, the ellipse is defined with the usual embedding function,
\BEA
\iota(\theta) = (\cos\theta,2\sin\theta)^\top, \quad\quad \theta\in[0,2\pi].\label{ellipseembedding}
\EEA
such that the Riemannian metric is given by a scalar component, $g_{11}(\theta) =  \sin^2\theta + 4\cos^2\theta$. For this example, $\mathcal{L}$ is not invertible since it has a zero eigenvalue with constant eigenfunction. For the problem to be well-defined, $f$ has to satisfy the second Fredholm condition. To ensure this solvability condition, we set the true solution to be $u(\theta)=\cos\theta$. With this function $u$, one can check that,
\BEA
f(\theta) &=& b(\theta) \cdot \nabla u(\theta) + \frac{1}{2} c(\theta) \nabla_1 \nabla_1 u(\theta) \nonumber\\
&=& b \frac{\partial u}{\partial \theta} g^{11} + \frac{1}{2}c(\theta)  \Big(\frac{\partial^2u }{\partial\theta^2}  - \Gamma^{1}_{11}  \frac{\partial u}{\partial\theta} \Big) \label{fellipse} \\
&=& -\sin\theta\cos\theta g^{11} + \frac{1}{2} (1.1 +\cos\theta) (-\cos\theta - 3g^{11}\sin^2\theta\cos\theta ),\nonumber
\EEA
where $g^{11} = 1/g_{11}$ is the inverse of the Riemannian metric, $g_{11}$ and $\Gamma^{1}_{11} = \frac{1}{2} g^{11} \frac{\partial g_{11}}{\partial\theta}$ is the Christoffel symbol of the second kind. So, the linear problem that solves for $u$ given $f$ in \eqref{fellipse} is in the range of $\mathcal{L}$ that has non-unique solutions (since $\cos\theta + d$ for any constant $d$ are also solutions).

In the top panel of Figure~\ref{ellipsefig2}, we plot the analytical $f$ in \eqref{fellipse}. In the same figure, we also plot the estimated $\hat L \vec{u}$, where components of $\vec{u}$ are evaluated on equally angle distributed points $\{\theta_i = i\frac{2\pi}{N}\}_{i=0,\ldots,N-1}$. In this numerical experiment, we set $N=1000$ and the number of $k$ nearest neighbor to be $k=200$. Based on the automated bandwidth estimation \cite{bh:15vb}, we found that $\epsilon = 10^{-4}$ is an adequate value for the  prototypical bandwidth parameter. In fact, the same value of $\te=10^{-4}$ will also be used in the Gaussian kernel, $h_{\te}$, that is used to estimate the sampling distribution, which will be used for the right normalization to compensate for the bias induced by nonuniform sampling distribution on the ellipse. Qualitatively, $\hat L\vec{u}$ and $\vec{f}$ are in good agreement. Quantitatively, the error in uniform norm is $\|\hat L\vec{u}-\vec{f} \|_\infty =  0.0082$. In the bottom panel of Figure~\ref{ellipsefig2}, we compare the estimated solution from the pseudo-inverse operation, $\hat L^\dagger \vec{f}$, with the analytical solution $\vec{u}$. Notice the good qualitative agreement; the error in uniform norm is $\|\hat L^\dagger \vec{f}-\vec{u} \|_\infty =  0.0049$.

Now, we consider only a half ellipse domain where the embedding function in \eqref{ellipseembedding} is defined only on $\theta \in [0,\pi]$. In this configuration, the solution that we are looking for, $u(\theta) = \cos\theta$, satisfies the homogenous Neumann boundary condition. In this numerical simulation, we keep the same value of parameters as in the full ellipse case. The pointwise operator estimation (as shown in the top panel of Figure~\ref{halfellipsefig3}) is accurate away from the boundary. The corresponding error in uniform norm, $\|\hat L\vec{u}-\vec{f} \|_\infty =  0.1634$, occurs near the zero boundary. In the bottom panel of Figure~\ref{halfellipsefig3}, the estimated solution based on the pseudo-inverse operation has error $\|\hat L^\dagger \vec{f}-\vec{u} \|_\infty =  0.0028$. 

\begin{figure}
\centering
\includegraphics[width = .7\textwidth]{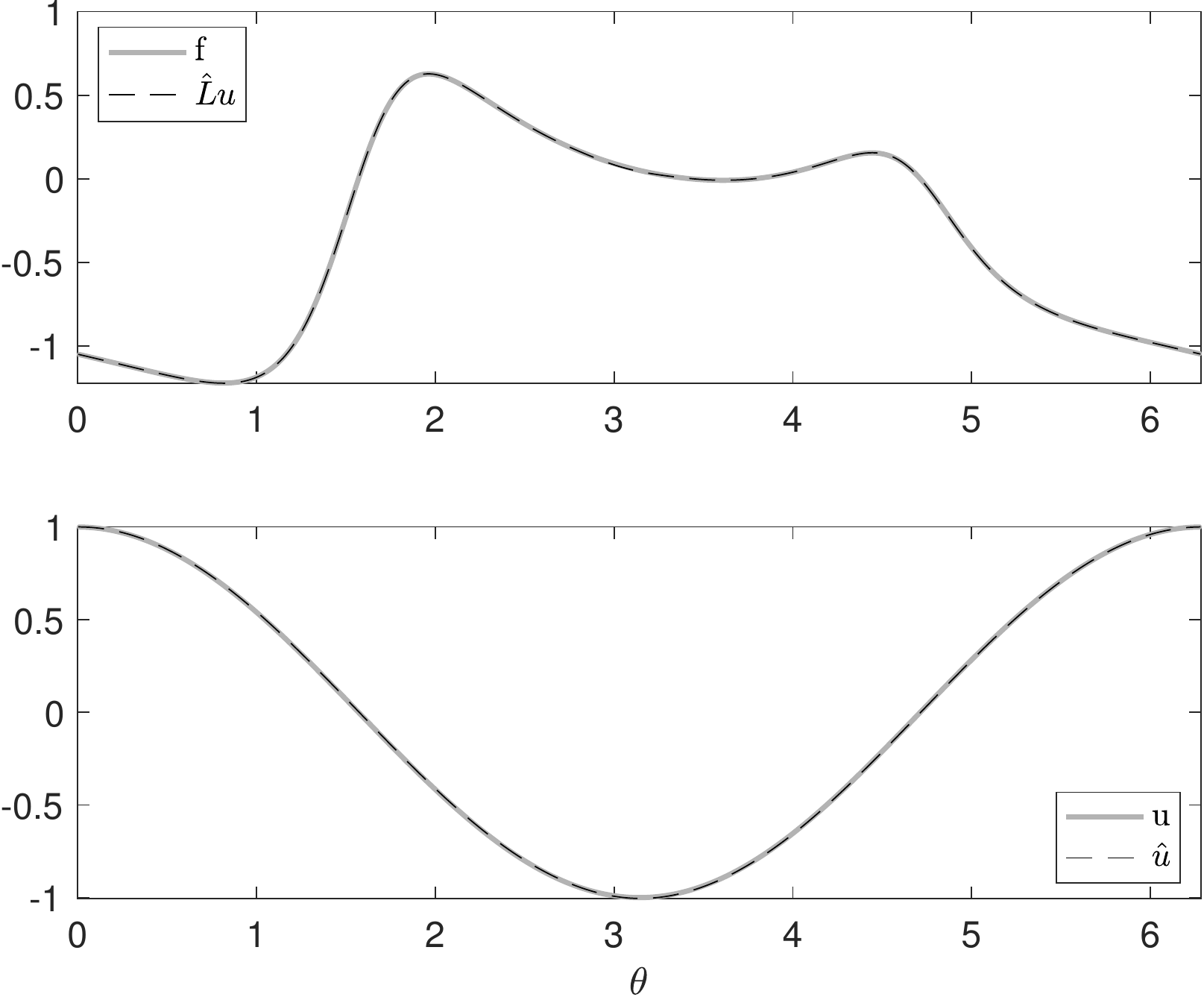}
\caption{Variable coefficients boundary value problem on a full ellipse: Pointwise operator estimation (top) and approximate solution by pseudo-inverse operation (bottom).}
\label{ellipsefig2}
\end{figure}

\begin{figure}
\centering
\includegraphics[width = .7\textwidth]{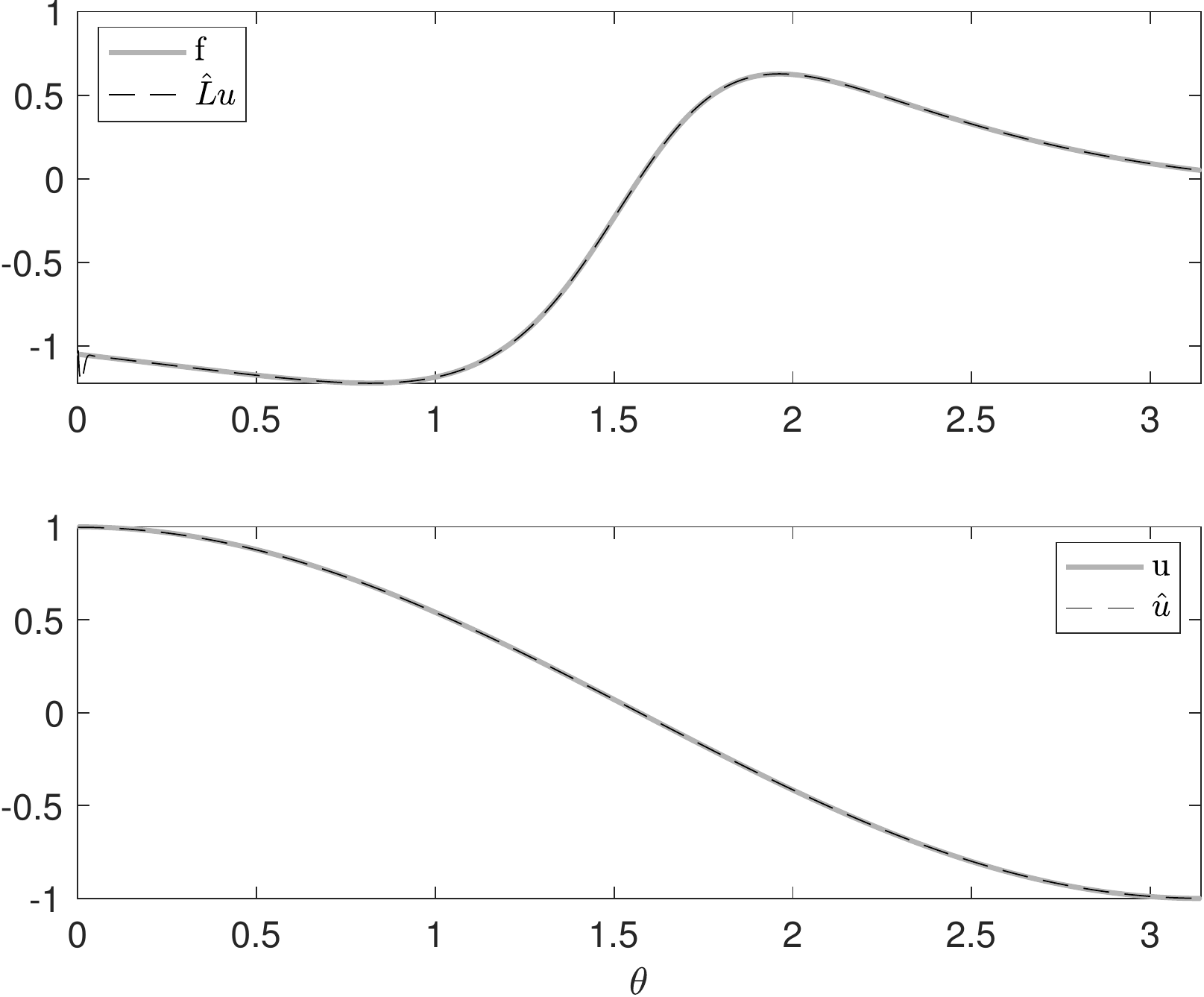}
\caption{Variable coefficients boundary value problem on a half ellipse: Pointwise operator estimation (top) and approximate solution by pseudo-inverse operation (bottom).}
\label{halfellipsefig3}
\end{figure}

\subsection{Variable coefficients differential equation on full and half tori}\label{2dexamples}

We consider solving an intrinsically two-dimensional boundary value problem in \eqref{PDE} with $a=0$ on a three-dimensional torus $\mathcal{M}\subset \mathbb{R}^3$ where the differential operator $\mathcal{L}$ is defined as in \eqref{generalL} with:
\BEA\label{torus_bc}
\begin{aligned}
b(\theta,\phi) &=  \begin{pmatrix} 2+\sin\theta \\ 0 \end{pmatrix}, \\
c(\theta,\phi) &= \begin{pmatrix} 3 +\cos\phi & 1/10 \\ 1/10 & 2 \end{pmatrix}.
\end{aligned}
\EEA
Here, the torus is defined with the standard embedding function:
\BEA\label{torusembedding}
\iota(\theta,\phi) = \begin{pmatrix} 
(2+\cos\theta)\cos\phi \\
(2+\cos\theta)\sin\phi \\
\sin\theta
\end{pmatrix}, \quad\quad \theta,\phi\in [0,2\pi]. 
\EEA

As in the previous example we design an analytic solution to this problem by setting $u(\theta,\phi) = \sin\theta\sin 2\phi$ and calculating $\ko u$. For this problem, it is easy to see that the Riemannian metric is
\BEA 
g_{(\theta,\phi)}(u,v) = u^\top \begin{pmatrix} 1 & 0 \\ 0 & (2+\cos\theta)^2 \end{pmatrix} v, \quad\quad \forall u,v\in T_{(\theta,\phi)}\mathcal{M} \cong\mathbb{R}^2\nonumber,
\EEA
and the only nontrivial Christoffel symbols of the second kind are
\BEA
\Gamma^2_{12} &=& -\frac{\sin\theta}{2+\cos\theta} \\
\Gamma^1_{22} &=& \sin\theta(2+\cos\theta).
\EEA
With this information, the explicit expression for $f$ is given by
\BEA
f &=& \mathcal{L}u = b\cdot \nabla u + \frac{1}{2} c_{ij} \nabla_i\nabla_j \nonumber\\
&=& g^{11}b^1 \frac{\partial u}{\partial \theta} + \frac{1}{2} c_{11} \frac{\partial^2 u}{\partial\theta^2} + c_{12}( \frac{\partial^2 u}{\partial\theta\partial\phi} - \Gamma^2_{12} \frac{\partial u}{\partial\phi})  \nonumber \\  
&& + \frac{1}{2}c_{22}( \frac{\partial^2 u}{\partial\phi^2} - \Gamma^1_{22} \frac{\partial u}{\partial\theta}).\label{torusf}
\EEA

In our numerical implementation, we choose a set of uniformly distributed grid points $\{\theta_i,\phi_j\}$ on $[0,2\pi]\times[0,2\pi]$, with $i,j=1,...,80$ points in each direction resulting in a total of $N=6400$ grid points. the matrix $\hat L$ is constructed with $k=128$ nearest neighbors and $\epsilon=0.0024$ obtained from the automated $\epsilon$-tuning algorithm given in \cite{bh:15vb} (and reviewed in the previous section). To compensate for the bias induced by non-uniform data points on the torus, we apply the right normalization (as discussed in Section~\ref{section33}) with sampling density estimated via a Gaussian kernel with bandwidth parameter $\te = 0.0179$, which is also estimated using the automated $\epsilon$-tuning algorithm of \cite{bh:15vb}. 

Figure~\ref{fulltorusfig4} shows the numerical estimates. Notice the agreement between $\hat L \vec{u}$ in panel (b) and the analytic $\vec{f}$ in panel (a). We also see the qualitative agreement of the estimated $\hat L^\dagger \vec{f}$ in panel (d) with the analytic $\vec{u}$ in panel (c). In Figure~\ref{fulltorusfig5}, we show the differences of panels (a) and (b) as well as (c) and (d), depicted as functions of intrinsic coordinates. Notice that the differences in the pseudo-inverse estimation from the analytical solution (right panel) are smaller than the operator estimation (left panel). In fact, the maximum errors are, $\|L\vec{u}-\vec{f} \|_\infty =  0.0361$, and $\|\hat L^\dagger \vec{f}-\vec{u} \|_\infty =  0.0076$, respectively.

\begin{figure}
\centering
\includegraphics[width = .75\textwidth]{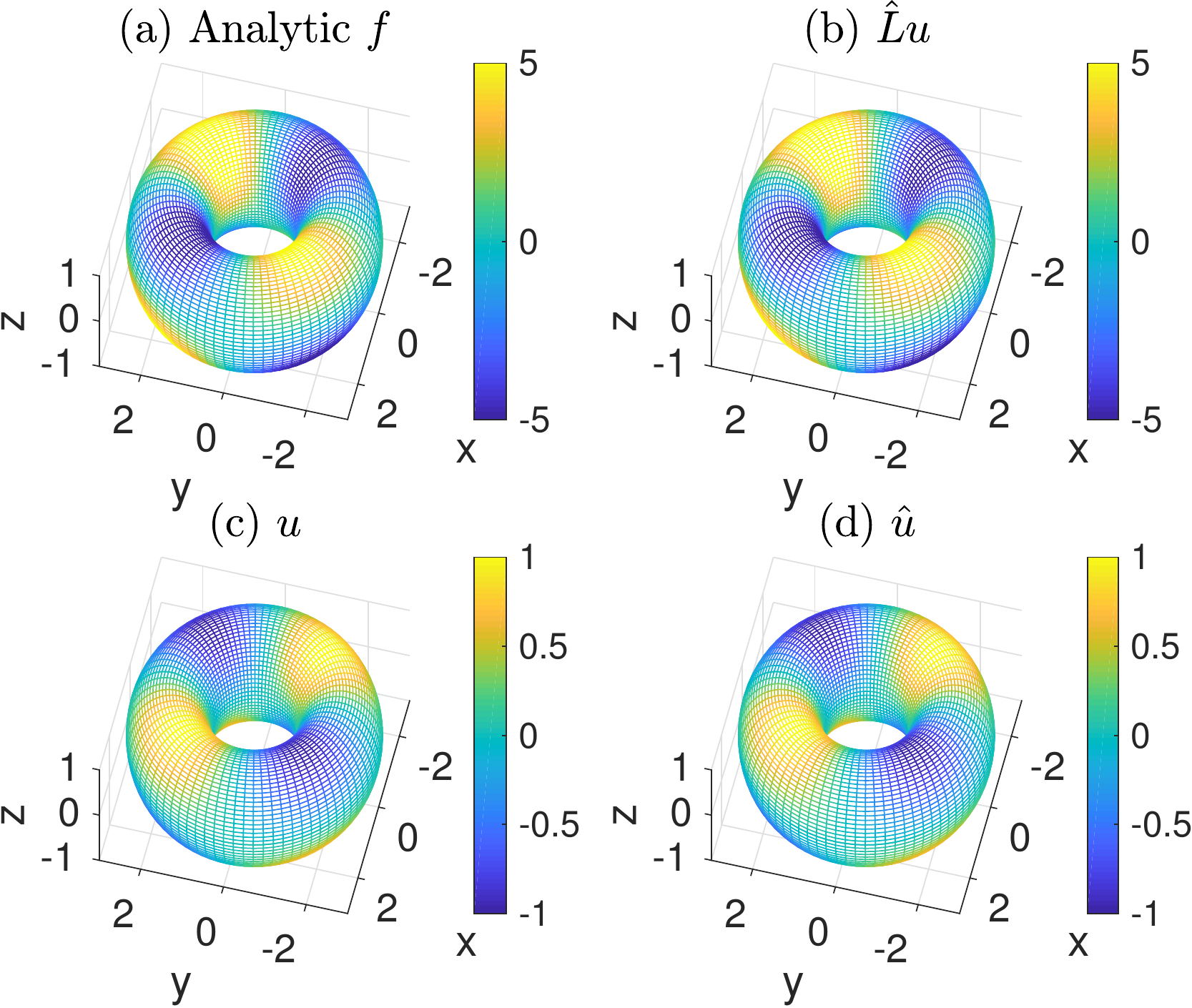}
\caption{Variable coefficients boundary value problem on a three-dimensional torus: (a) Analytic $f$ in \eqref{torusf}; (b) Pointwise operator estimation $\hat L \vec{u}$; (c) True solution $\vec{u}$; (d) Approximate solution via the pseudo-inverse operation, $\hat L^\dagger \vec{f}$.}
\label{fulltorusfig4}
\end{figure}

\begin{figure}
\centering
\includegraphics[width = .47\textwidth]{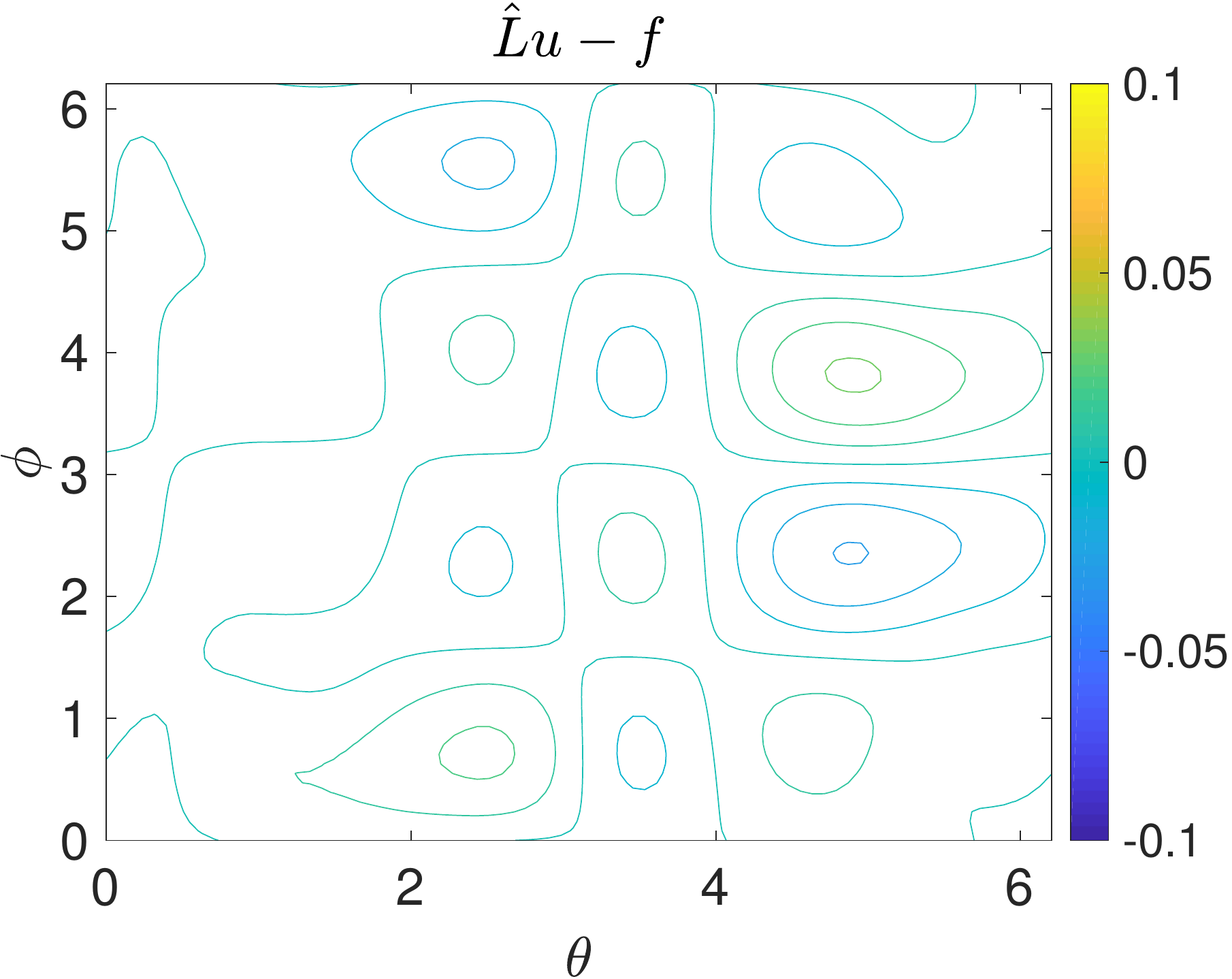}
\includegraphics[width = .47\textwidth]{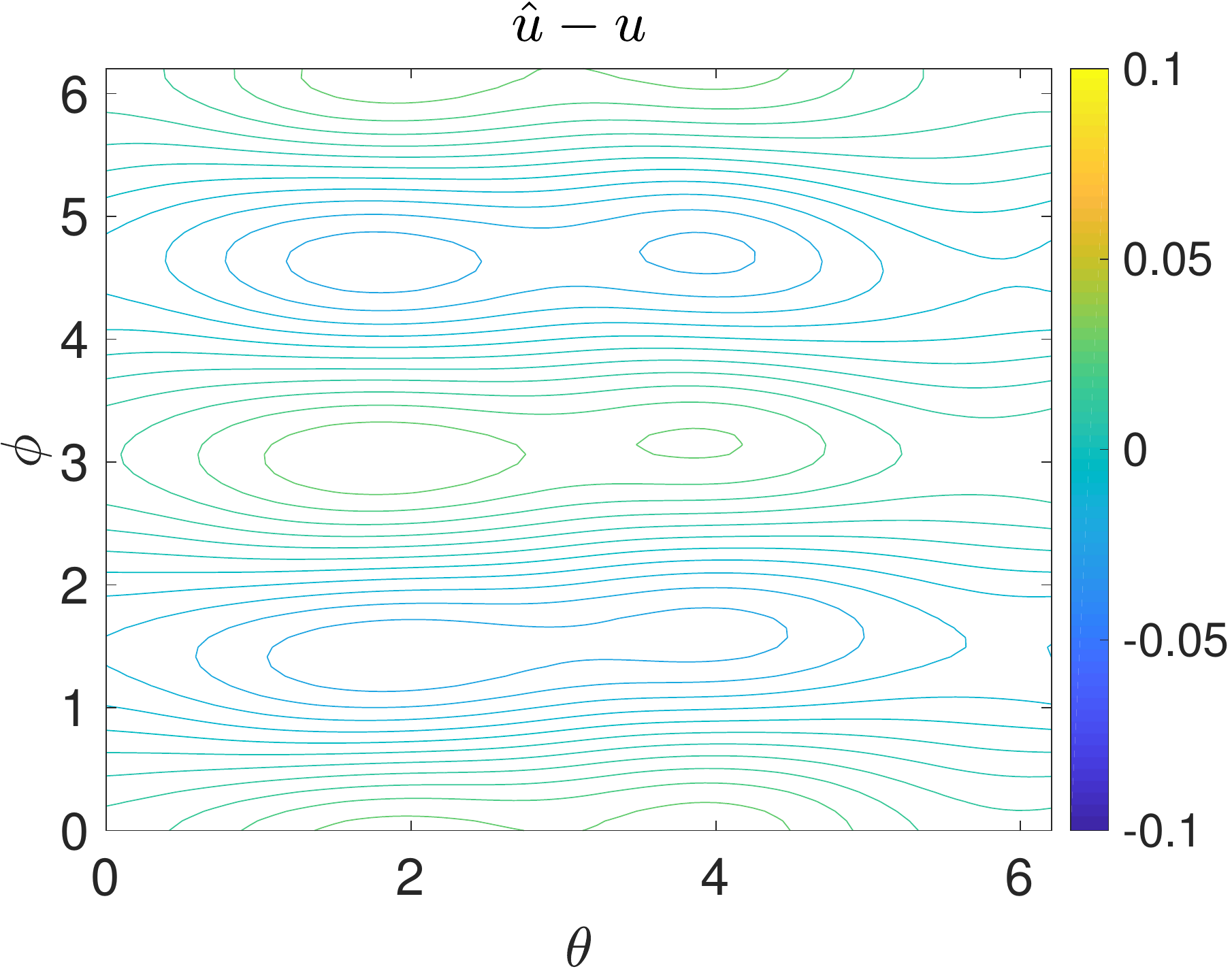}
\caption{Variable coefficients boundary value problem on a three-dimensional torus: Differences in the operator estimation (left) and the pseudo-inverse estimation (right)}
\label{fulltorusfig5}
\end{figure}

We also include a numerical simulation with a half torus. In this case, the manifold is defined via the embedding function in \eqref{torusembedding} with $\phi\in[0,\pi]$. For this experiment, we use the same $b$ and $c$ as in \eqref{torus_bc}. The only difference in the numerics is that the number of grid points corresponding to the $\phi$ coordinate is only 40, resulting in a total of $N=3200$ grid points. Fixing $k=128$ as before, the estimated epsilons are $\epsilon=0.0026$ and $\te=0.0179$. To satisfy the Neumann boundary conditions, we change the analytical solution to $u(\theta,\phi) = \sin\theta\cos 2\phi$. Figures~\ref{halftorusfig6} and \ref{halftorusfig7} show the numerical estimates compared to the corresponding analytical solutions. In this case, notice the larger errors near the boundaries. Overall, the quality of the solutions degrade compared to the full torus example above with errors $\|\hat L\vec{u}-\vec{f} \|_\infty =  0.8008$, and $\|\hat L^\dagger \vec{f}-\vec{u} \|_\infty =  0.0774$, respectively.

\begin{figure}
\centering
\includegraphics[width = .75\textwidth]{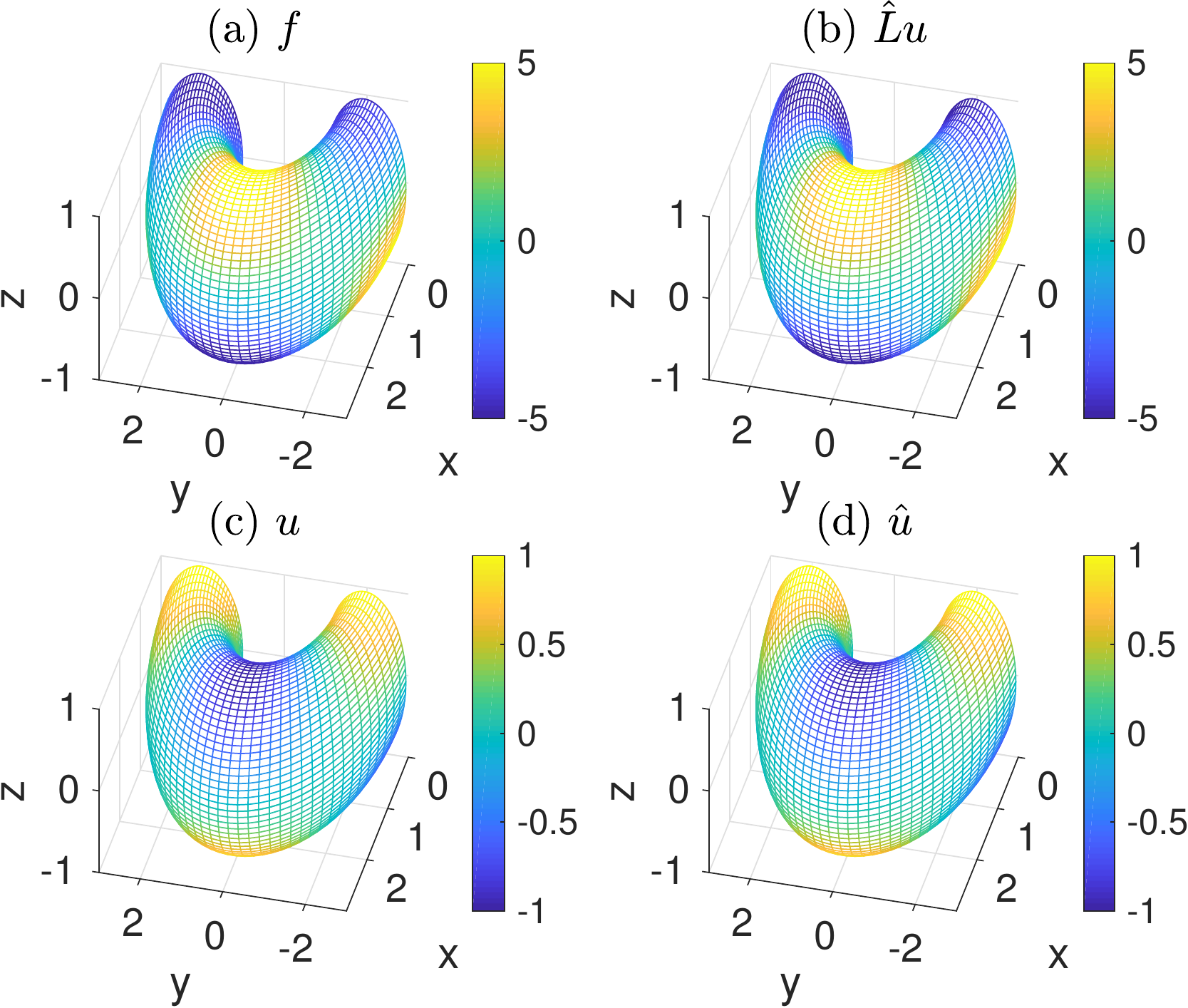}
\caption{Variable coefficients boundary value problem on a half three-dimensional torus: (a) Analytic $f$ in \eqref{torusf}; (b) Pointwise operator estimation $\hat L_\epsilon \vec{u}$; (c) True solution $\vec{u}$; (d) Approximate solution via the pseudo-inverse operation, $\hat L_\epsilon^\dagger \vec{f}$.}
\label{halftorusfig6}
\end{figure}

\begin{figure}
\centering
\includegraphics[width = .47\textwidth]{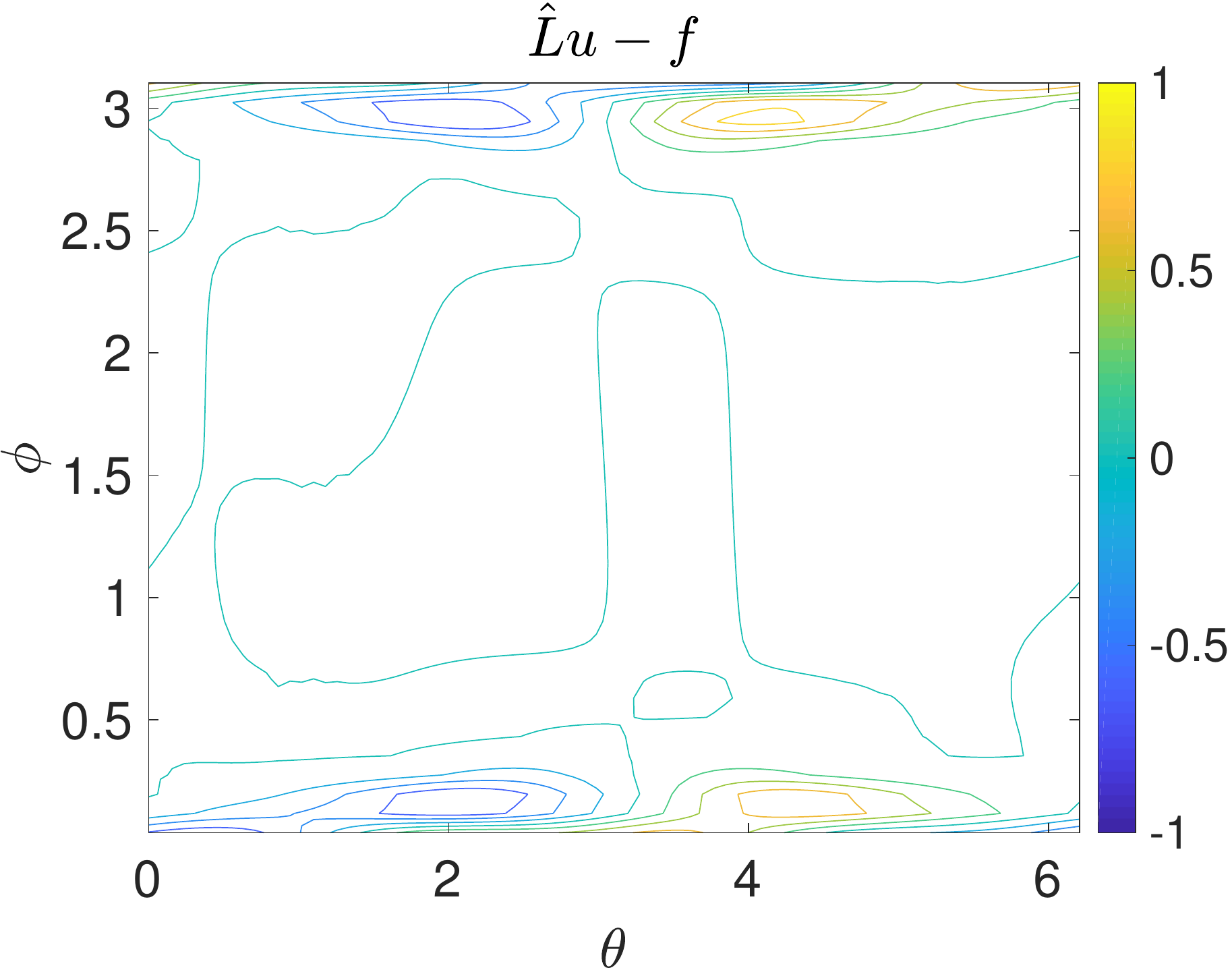}
\includegraphics[width = .47\textwidth]{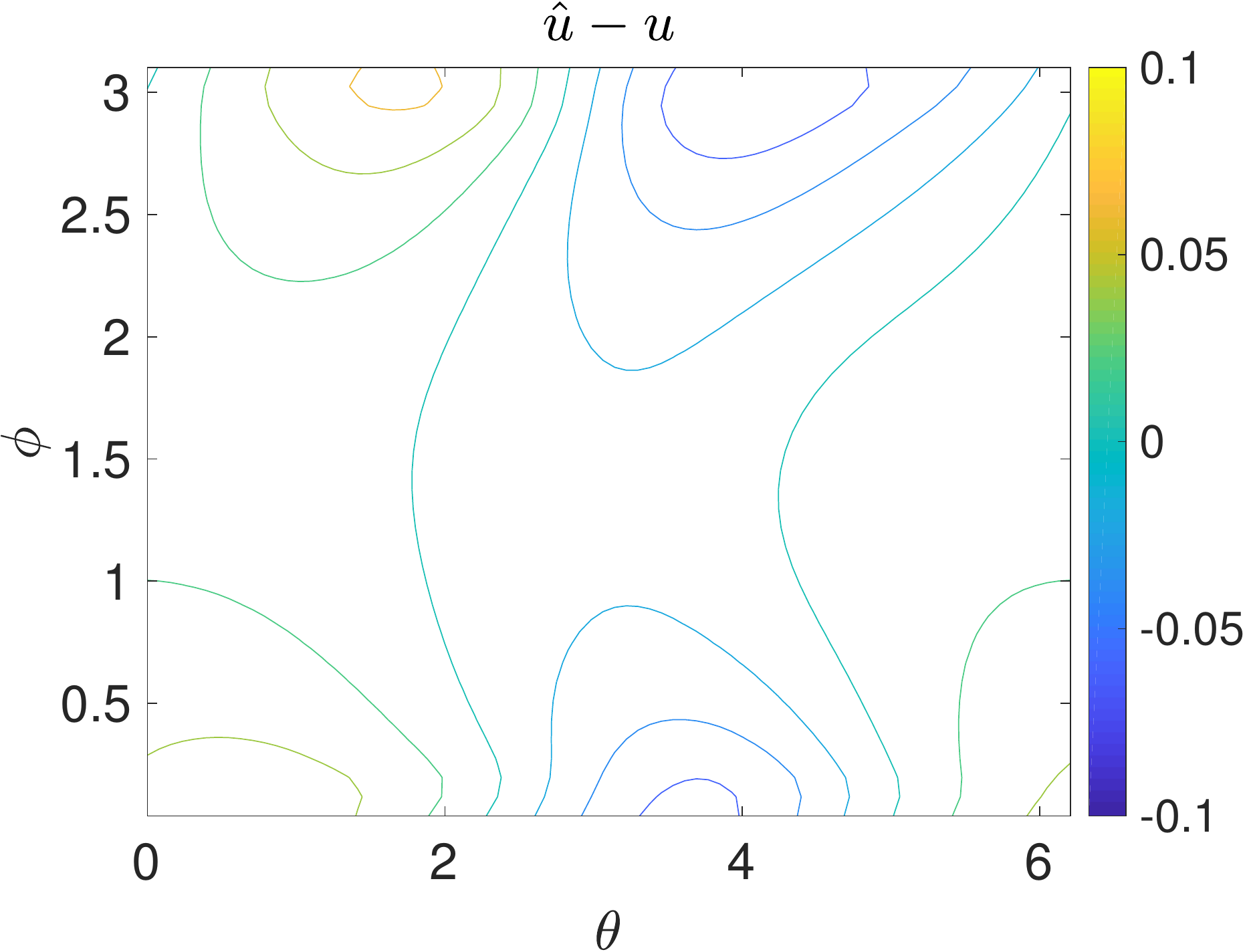}
\caption{Variable coefficients boundary value problem on half three-dimensional torus: Differences in the operator estimation (left) and the pseudo-inverse estimation (right)}
\label{halftorusfig7}
\end{figure}

\subsection{Numerical comparison with the mesh-less radial basis function method}
In this section, we compare the local kernel method with a mesh-less radial basis function (RBF) method \cite{fornberg2015solving} on the four examples in Sections~\ref{1dexamples} and \ref{2dexamples}. For these examples, we simply employ the radial basis functions defined on the intrinsic coordinates since the Riemannian metrics are explicitly known. For the one-dimensional examples, the implementation simply applies the collocation method on equation \eqref{fellipse} (along with the corresponding boundary conditions) with $u$ defined as,
\BEA
u(\theta) = \sum_{k=1}^{N} \lambda_k \Phi (\|\theta_k -\theta\|).
\EEA 
For the two-dimensional examples in Section~\ref{2dexamples}, the RBF functional approximation is defined over $(\theta,\phi)$. In the case of the half torus, the Neumann boundary condition is given as $\frac{\partial f}{\partial\vec{n}}(x) = \nabla f \cdot \vec{n} = \frac{\partial f}{\partial \phi} = 0$ for $x\in \partial \mathcal{M}$. In our numerics, we implement the discrete approximation using Kansa's formulation with Gaussian function $\Phi(r) = \exp(-(sr)^2)$ with shape parameter $s$. We employ the discretization exactly on the same grid points we used in the previous two sections, namely $N=1000$ points in both one-dimensional examples and $N=6400/N=3200$ grid points on the full/half tori examples, respectively. For the numerical results shown in this section, we used the MATLAB RBF toolbox developed by Scott Sarra \cite{sarra:17}. To obtain the weight $\lambda_k$ for these boundary value problems, we used the pseudo-inverse operation as in the local kernel method. For these examples, we should pointed out that the local kernel errors in Figure~\ref{RBFcomparison} are slightly lower than the results presented in Sections~\ref{1dexamples} and \ref{2dexamples} since the kernel bandwidth parameters $\epsilon$ and $\te$ are empirically tuned by comparing the resulting estimates to the analytical solutions.   

\begin{figure}
\centering
\includegraphics[width = .47\textwidth]{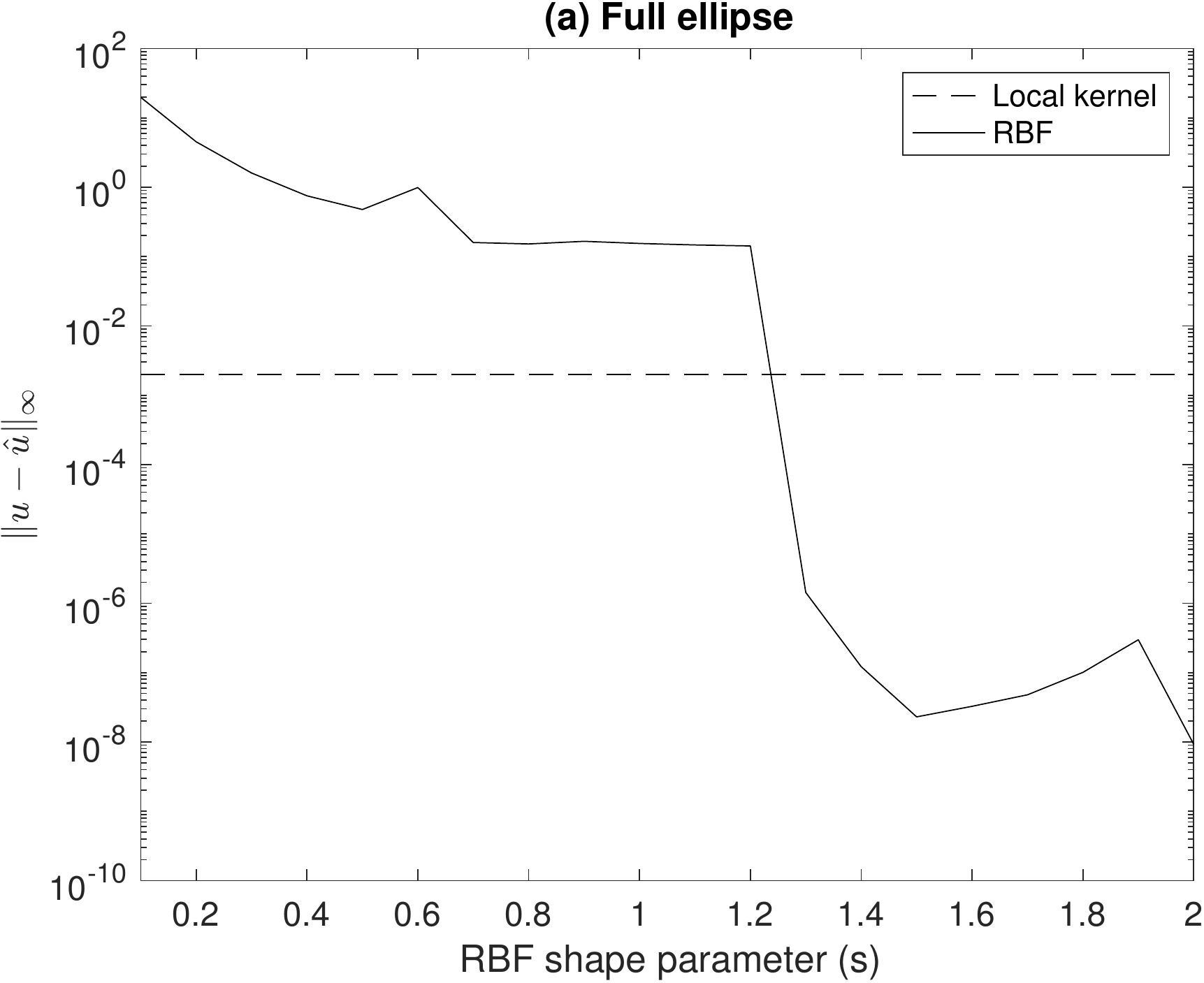}
\includegraphics[width = .47\textwidth]{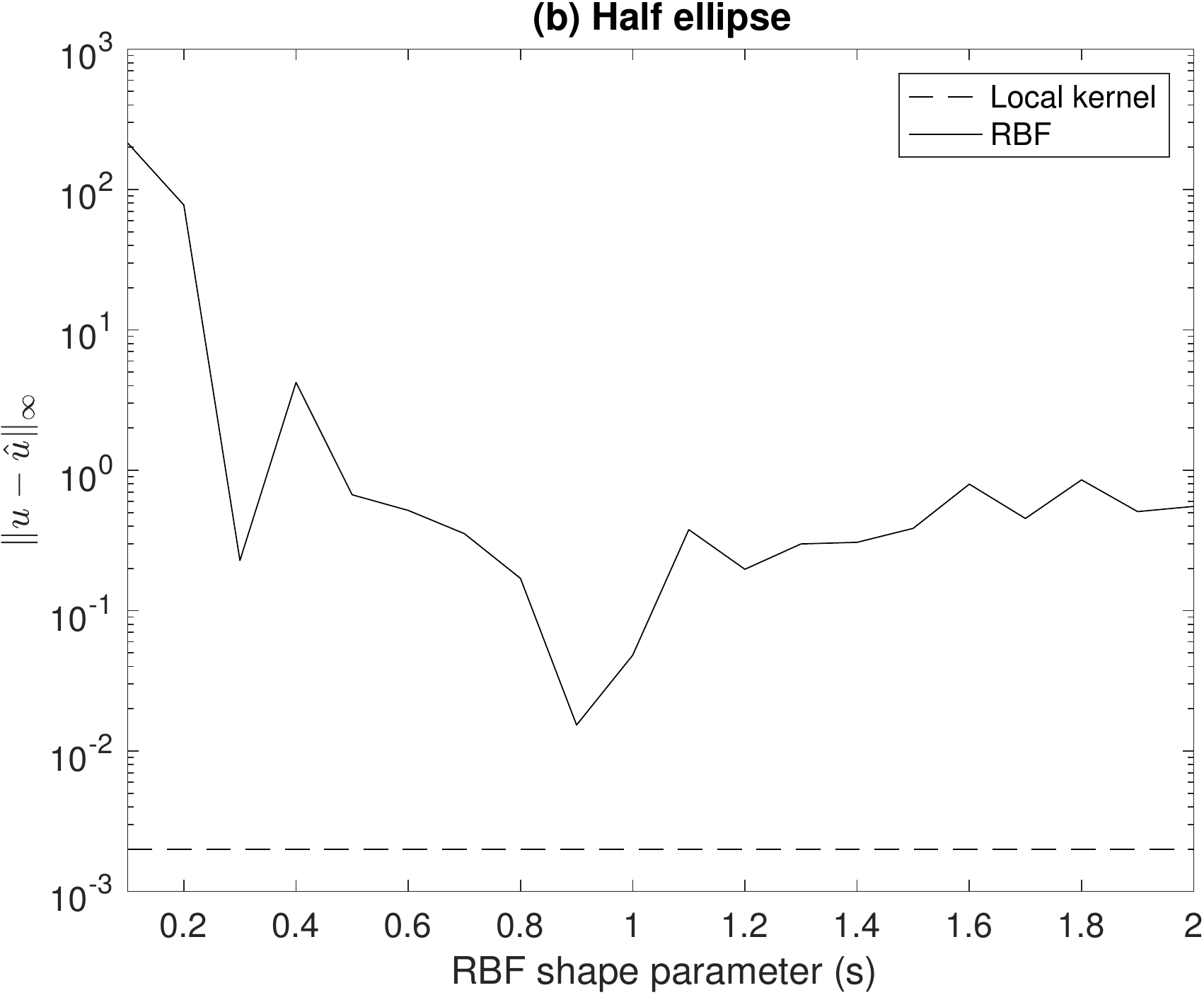}
\includegraphics[width = .47\textwidth]{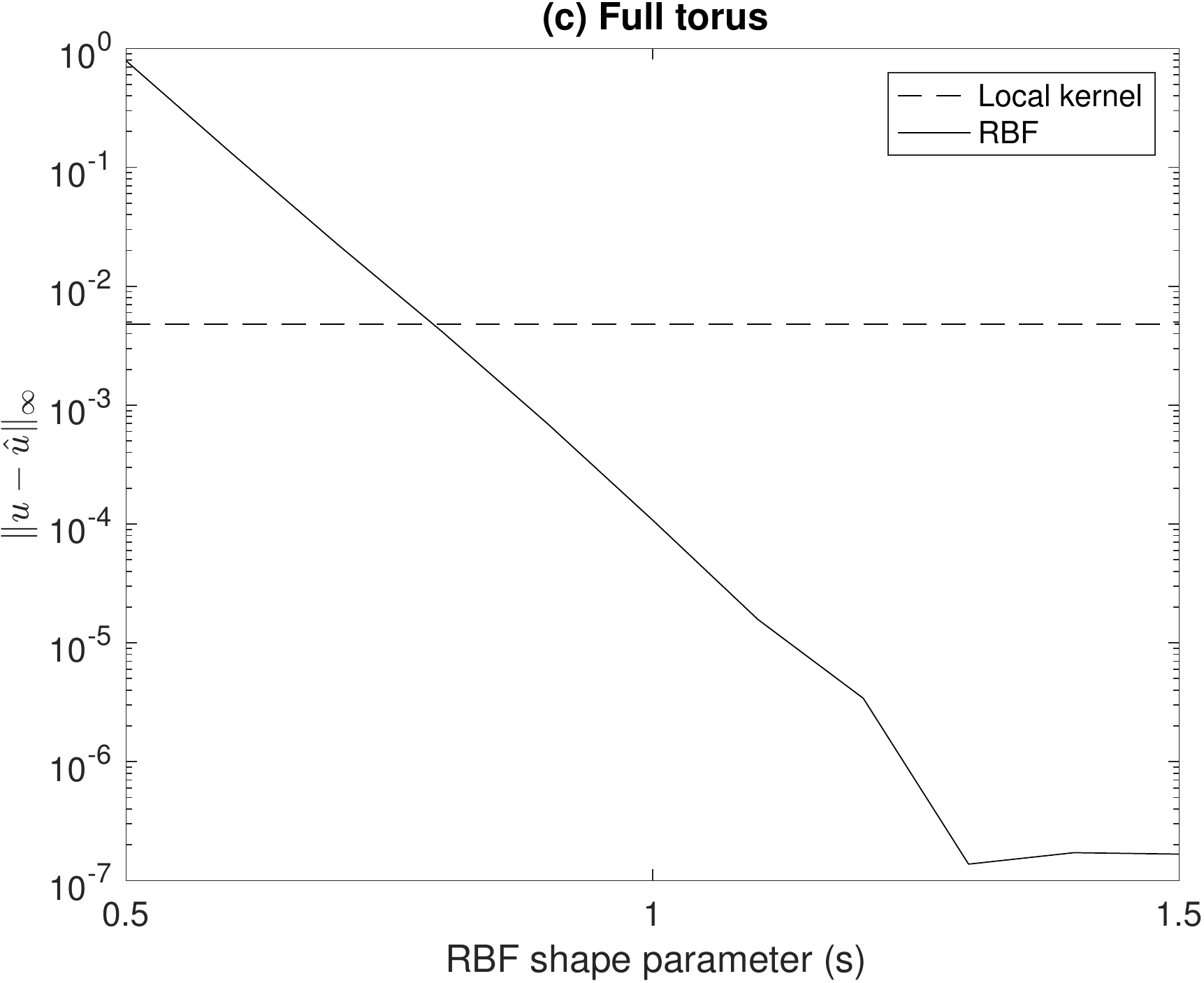}
\includegraphics[width = .47\textwidth]{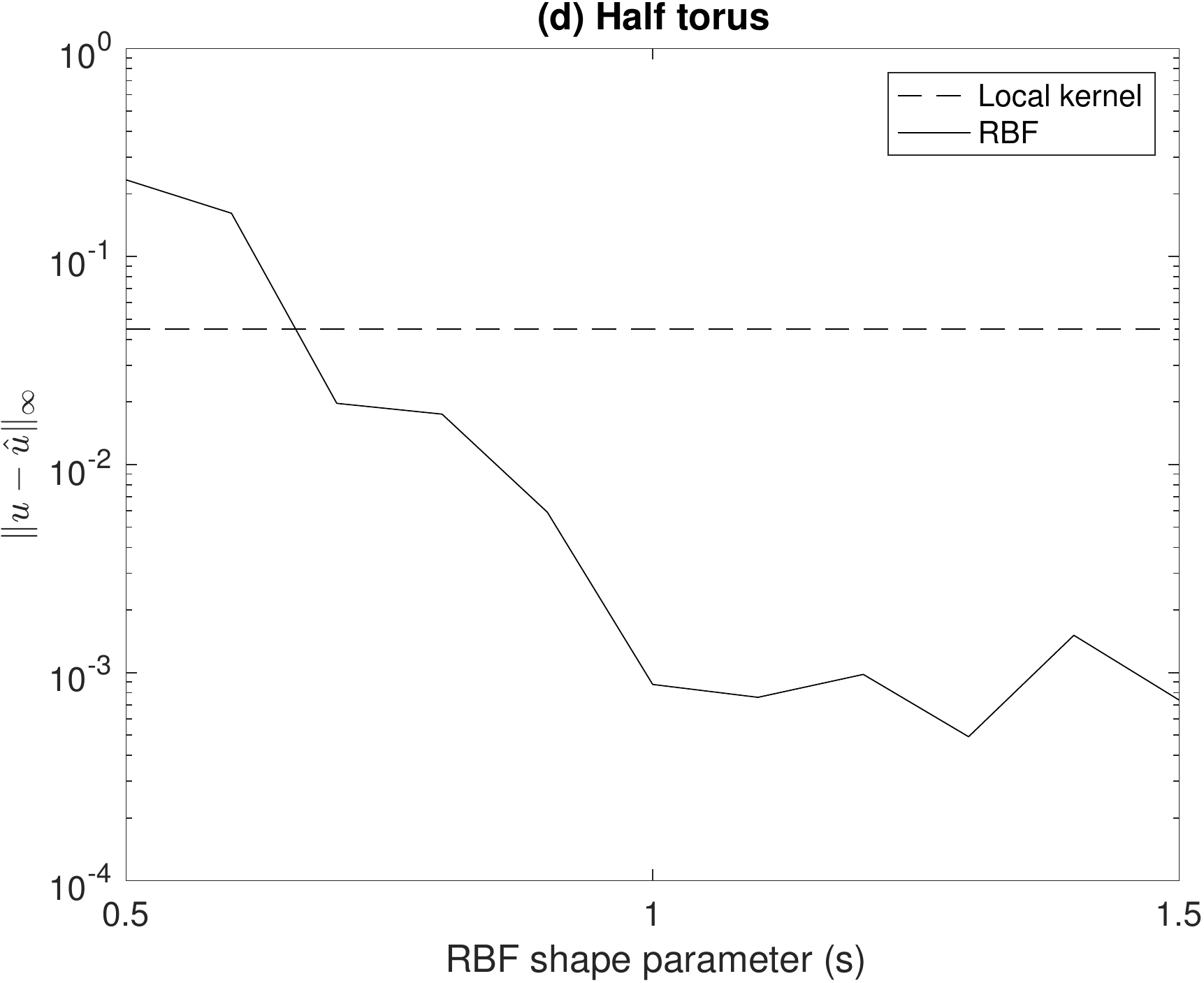}
\caption{Comparison with the RBF collocation method: The RBF errors are plotted as functions of the shape parameter, $s$.}
\label{RBFcomparison}
\end{figure}

From Figure~\ref{RBFcomparison}, it is clear that RBF is superior to the local kernel approach when the shape parameter $s$ is appropriately tuned, except in the half ellipse case. For the half ellipse case, similar results are also found using inverse-quadratic basis function and even using less grid points (results not reported). While the local kernel method is less accurate than RBF in the three cases shown above, extending the RBF method to arbitrary manifolds embedded in $\mathbb{R}^n$ requires significant modifications. For surfaces in $\mathbb{R}^3$, one approach proposed by \cite{piret2012orthogonal} is to first approximate the manifold as an iso-surface of a distance function constructed using the RBF method. Subsequently, the tangential derivatives are approximated by either projecting the ambient three-dimensional derivatives onto the tangent space of the manifold or via the orthogonal gradient method. On the other hand, the local kernel technique extends naturally on arbitrary manifolds as shown in the next section.

\subsection{An example on a manifold with unknown embedding}

In this section, we apply the the local kernel method to solve  
\BEA
\Delta u(x)  &=f(x), \ \ \ \ x = (x_1,x_2,x_3)  \in \mathcal{M} 
 \label{SEQ}
\EEA
where  $f(x)=x_1x_2$ and the surface $\mathcal{M}\subset \R^3$ is a two dimensional-closed manifold that is homeomorphic to the unit sphere, $S^2$, with unknown embedding functions. The surface used in this section is from Keenan Crane's 3D repository \cite{crane}. We neglect the RBF method here since a proper implementation for arbitray manifold (such as, using the technique proposed in \cite{piret2012orthogonal}) requires a significant algorithmic modification compared to the basic RBF formulation for solving PDE's in \cite{fornberg2015solving}, which is beyond the scope of this paper. For comparison, we provide numerical estimates using finite element method (FEM). Numerically, we used FELICITY, an FEM toolbox for Matlab \cite{Walker}.

In Figure~\ref{Spot}, we compare the local kernels solution of \eqref{SEQ} with the FEM solution. The maximum absolute error between the two solutions was $.0064$ and the tuned bandwidth parameter for the local kernel is $\epsilon=.002$ .  To compute the FEM solution, we provided FELICITY with the triangulated mesh of the surface, which consisted of $2930$ points and a connectivity matrix for the triangle elements. In this case, the analytic solution is not known and, since we have no way of obtaining more points on the surface, we restricted the FEM algorithm to use a linear finite element space. 

This example suggests that even when the embedding function of a  manifold is not known, the local kernel technique can be used to approximate the solution to \eqref{PDE} when all the relevant information is specified in ambient coordinates.

\begin{figure}
\centering
\includegraphics[width = .58\textwidth]{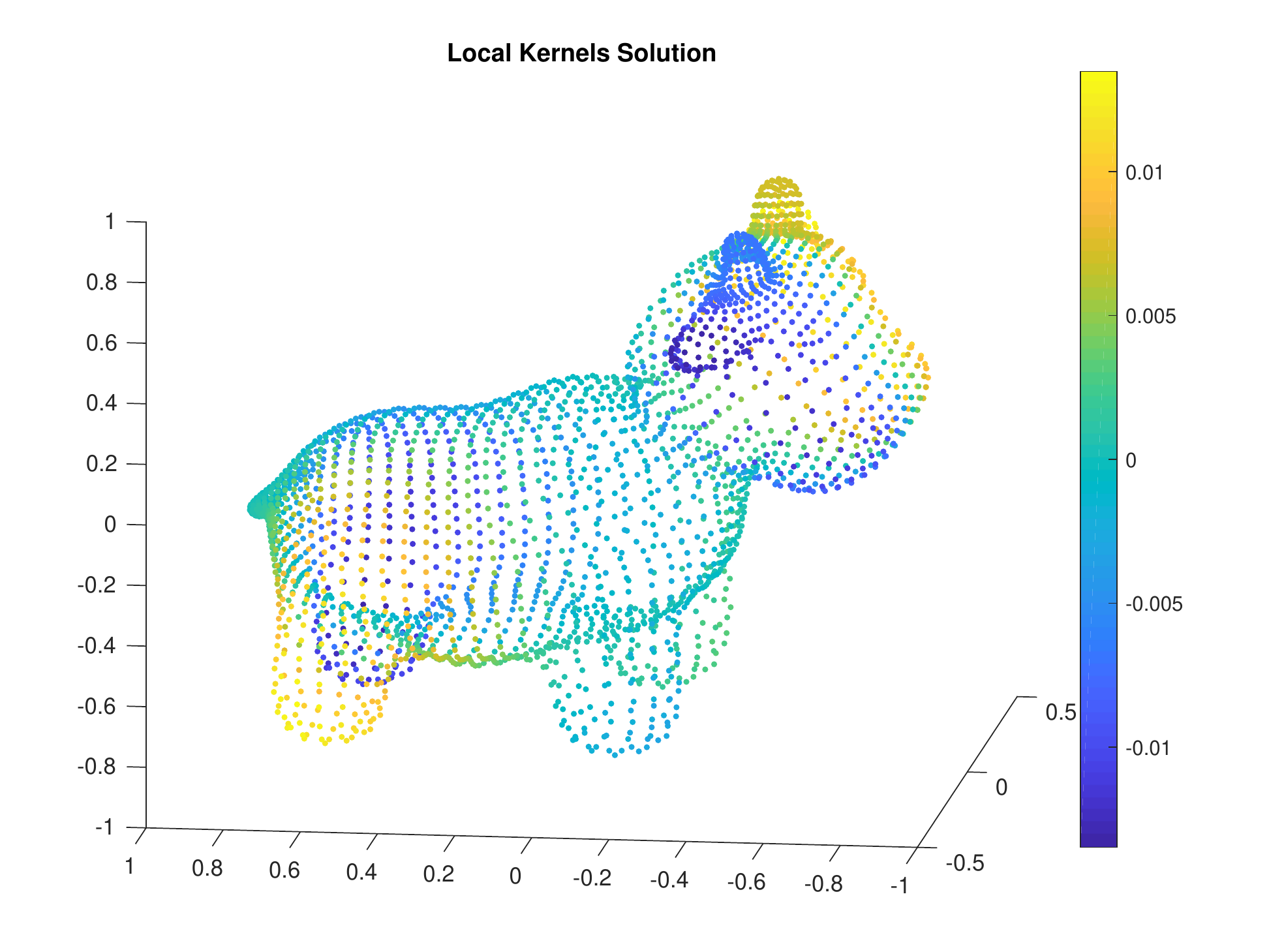}
\includegraphics[width = .58\textwidth]{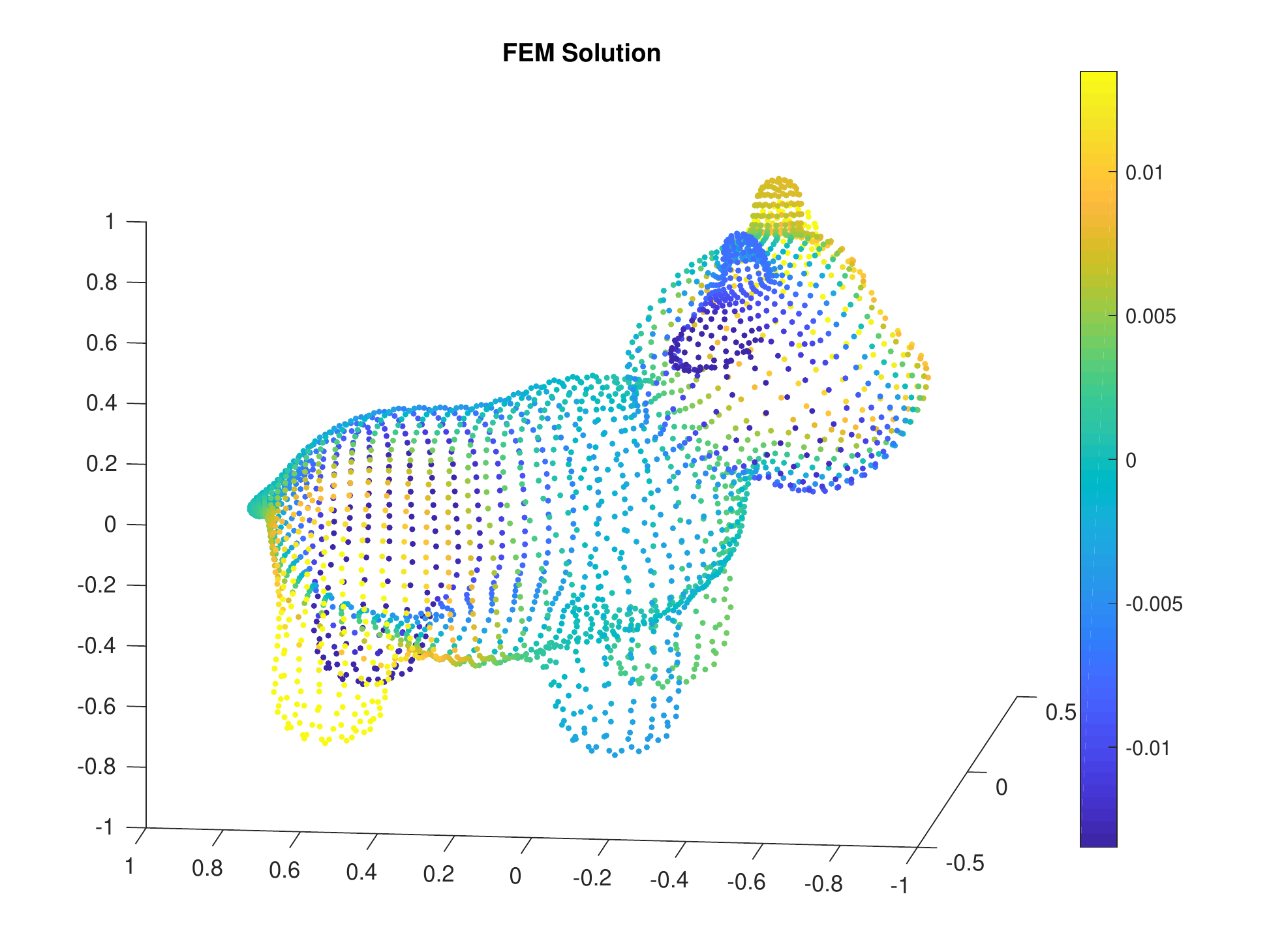}
\includegraphics[width = .58\textwidth]{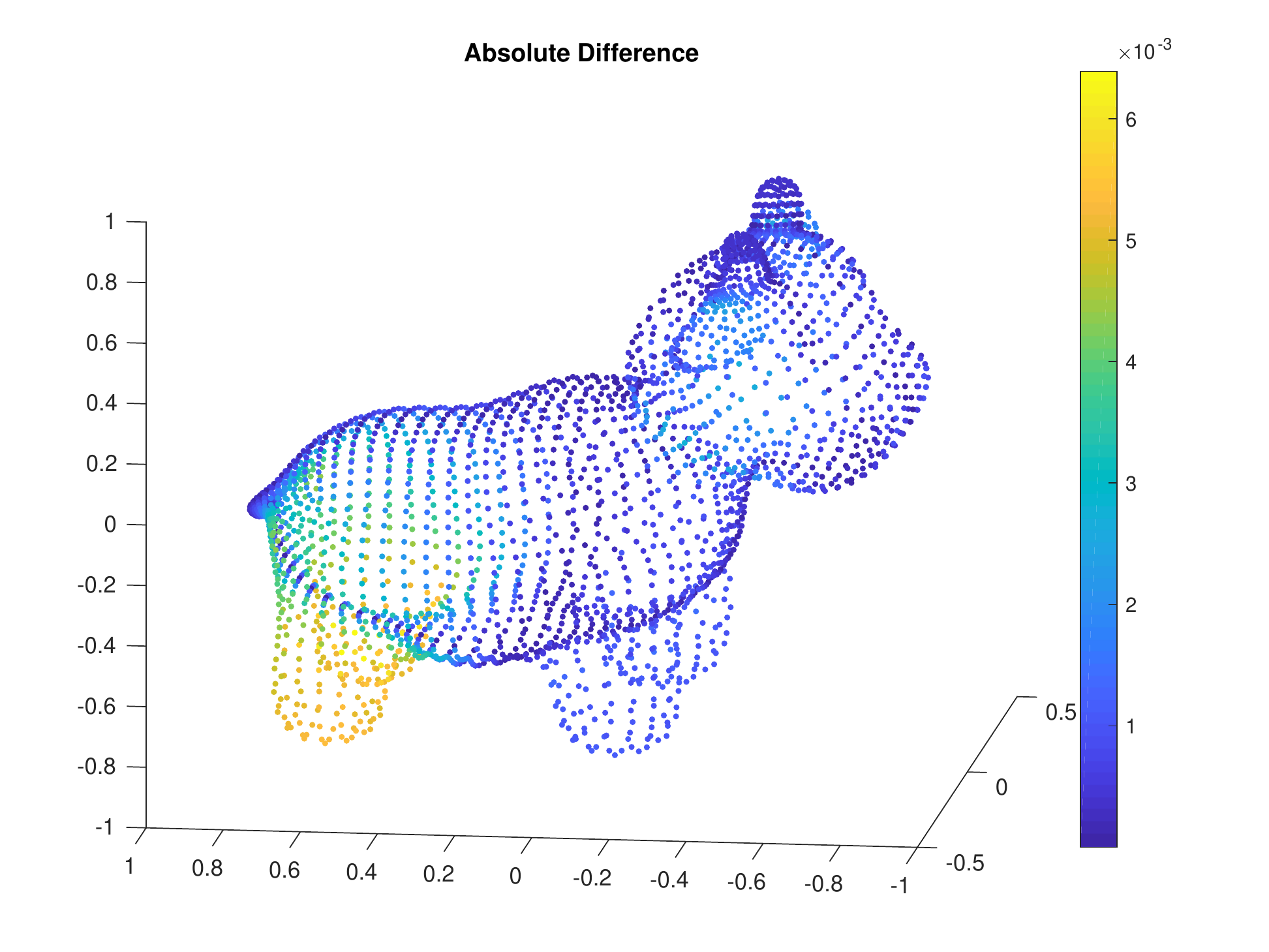}
\caption{Numerical approximation of \eqref{SEQ} using (a) Local Kernel and (b) FEM from FELICITY. Panel (c) shows the absolute difference between the two approximate solutions.} 
\label{Spot}
\end{figure}

\section{Summary and discussion}
In this paper we used the local kernel method, a recently developed generalization of the diffusion maps algorithm \cite{bs:16}, to approximate solutions of linear elliptic PDE's on compact Riemannian manifolds (with Neumann boundary condition if the manifold has boundaries). Theoretically, we show the convergence of the approximate solution under the classical well-posedness conditions of the boundary value problem \eqref{PDE}. Furthermore, when the linear problem has non-unique solutions, we showed that the minimum norm solution solves the PDE in \eqref{PDE} under the Fredholm alternative's second solvability condition.  Numerically, we tested the local kernel technique on various variable coefficient linear PDE's on flat and non-flat manifolds with known embedding functions, including a closed interval, full and half ellipse, and two-dimensional full and half torus. In these instructive examples, where the analytical solution was known, we found that the estimated solution was more accurate than the operator estimation. The theoretically established order $\epsilon$ convergence rate was also numerically verified. From the numerics, we also found that the convergence rate in terms of the number of grid points, $N$, was much sharper than our theoretical estimate. Additionally, we tested our method on a two-dimensional closed manifold homeomorphic to a sphere with unknown embedding function. In this case, we found that our approximate solution was close to the solution obtained from the FEM. 

While  these results are encouraging, this approach has several limitations. Most notably, it is unclear how to extend this technique to approximate nonlinear differential operators. Further investigation is also required to extend this technique to incorporate non-Neumann boundary conditions. The local kernel method is also less accurate at the boundary of a manifold and is consequently more suited for application on closed manifolds. Finally, while the local kernel algorithm is simple to implement, this method is an order-one scheme and it is unclear how to increase the accuracy. 

\section*{Acknowledgments}
The authors would like to thank the reviewer for mentioning the relevant meshless methods: the RBF and point integral method. The authors also thank Alfa Heryudono for pointing us to Scott Sarra's MATLAB radial basis function toolbox \cite{sarra:17}.  The research of J.H. is partially supported by the ONR Grant N00014-16-1-2888 and NSF Grant DMS-1619661. F. G. is supported as a GRA under these grants.

\appendix
\section{Proof of Lemma~\ref{lemma5}} 

The terms $\mathcal{O}(\epsilon,\te)$ in Lemma~\ref{lemma5} are the errors of the continuous operator $L_{\epsilon,\te}$, deduced in Proposition~\ref{prop4}. The term $\mathcal{O}\big(\frac{q(x_i)^{1/2}}{\sqrt{N}\te^{2+d/4}}\big)$ is the sampling error for a discrete approximation,
\BEA
\hat{q}_{\te} (x_i):= \frac{\te^{-d/2}}{N} \sum_{i=1}^N h(\te,x_i,x_j) = H_{\te}q(x_i) + \mathcal{O}(\te^2) = q_{\te}(x_i) +  \mathcal{O}(\te^2).\nonumber
\EEA
Since this error rate is simply a special case of the result in the Appendix~B.1 of \cite{bh:15vb} for a fixed bandwidth Gaussian kernel, we will not repeat it here.
  
For the remaining of this Appendix, we will compute the last error term in Lemma~\ref{lemma5}, which is the bias in estimation of $\mathcal{L}$ using the discete approximation $\hat{L}_{\epsilon,\te}$. To achieve this, we first define the following random variables:
\BEA
\nonumber
F_i(x_j)=\frac{K(\epsilon,x_i,x_j)u(x_j)}{\hat{q}_{\te}(x_j)} \ \  \ \ 
G_i(x_j)=\frac{K(\epsilon,x_i,x_j)}{\hat{q}_{\te}(x_j)} 
\EEA

For $A\in L^1(\mathcal{M},q)$, we define
\BEA
\mathbb{E}(A) : = \int_{\mathcal{M}} A(x) q(x) dx.
\EEA
Then, we can write, 
\begin{align*}
L_{\epsilon,\te}u(x_i):= \frac{1}{\epsilon}\left(\tilde{q}_{\epsilon,\te}^{-1}G_{q,\epsilon}(u(x_i)q^{-1}_{\te}(x_i))-u(x_i)\right)&=\frac{1}{\epsilon}\left(\frac{\mathbb{E}(F_i)}{\mathbb{E}(G_i)}-u(x_i)\right)
\end{align*}
and its discrete approximation,
\BEA
\hat{L}_{\epsilon,\te}u(x_i):= \frac{1}{\epsilon}\left(\frac{\sum_j F_i(x_j)}{\sum_j G_i(x_j)}-u(x_i)\right),\nonumber
\EEA
in terms of $F_i$ and $G_i$. Since the error of neglecting the $j$th component in each of the sum is negligible, $\mathcal{O}(N^{-1}\te^{-d/2})$, we will proceed by neglecting the $i$th term in the sum. Now that the sum is i.i.d., by 
law of large number, we expect,
\BEA
\frac{\sum_{j \neq i}F_i(x_j)}{\sum_{j\neq i}G_i(x_j)} \to \frac{\mathbb{E}(F_i)}{\mathbb{E}(G_i)}. \nonumber
\EEA
as $N\to\infty$.

Following the approach in \cite{singer2006graph,bh:15vb}, we compute,
\BEA
P(|L_{\epsilon,\te}u(x_i)- \hat{L}_{\epsilon,\te}u(x_i) |> a) &\approx & P\left(\frac{1}{\epsilon}\left(\frac{\mathbb{E}(F_i)}{\mathbb{E}(G_i)}-\frac{\sum_{j \neq i}F_i(x_j)}{\sum_{j\neq i}G_i(x_j)}\right)>a\right) \nonumber \\ &=& 
P\left(\sum_{j \neq i}Y_j>a\epsilon(N-1)\mathbb{E}(G_i)^2\right), \label{discreteerrorLee}
\EEA
where 
$Y_j=\mathbb{E}(G_i)F_i(x_j)-\mathbb{E}(F_i)G_i(x_j)+a\epsilon\mathbb{E}(G_i)\left(\mathbb{E}(G_i)-G_i(x_j)\right)$. Note that the approximation in the first line is due to neglecting the $j$th component in each of the sum. Using the Chernoff bound, 
\BEA
P\left(\sum_{j \neq i}Y_j>a\epsilon(N-1)\mathbb{E}(G_i)^2\right)&\leq 2\exp\bigg(\frac{-a^2\epsilon^2(N-1)^2\mathbb{E}(G_i)^4}{4(N-1)\mathrm{Var}(Y_j)}\bigg) \label{chernoff}
\EEA
our next task is to express the upper bound in terms of $\epsilon$ and $N$ by computing $\mathrm{Var}(Y_j)$ and $\mathbb{E}(G_i)^4$. Based on the expansion in \eqref{Gqe}, we deduce
\begin{align*}
\mathbb{E}(F_i)&=\epsilon^{d/2}\gq\left(uq_{\te}^{-1}\right)\\
&=\epsilon^{d/2}m\tmo^{-1} u\left(1-\te \tm\tom -\te\tm q^{-1}\Delta q+\epsilon m^{-1}\omega +\epsilon\frac{\ko(u)}{u}\right)\\ 
&\hspace{.5 in} +\epsilon^{d/2}\mathcal{O}\left(\epsilon\te,\epsilon^2,\te^2\right)\\
\mathbb{E}(F_i)^2&=\epsilon^d m^2\tmo^{-2} u^2\left(1-2\te \tm \tom-2\te \tm q^{-1} \Delta q+2 \epsilon m^{-1} \omega+2\epsilon \frac{\ko(u)}{u}\right)\\
 &\hspace{.5 in}+\epsilon^{d/2}\mathcal{O}\left(\epsilon\te,\epsilon^2,\te^2\right)\\
\mathbb{E}(G_i)&=\epsilon^{d/2} \tilde{q}_{\epsilon,\te}\\
&=\epsilon^{d/2}m\tmo^{-1} \left(1- \te \tm\tom-\te \tm q^{-1}\Delta q+\epsilon m^{-1} \omega\right)+\epsilon^{d/2}\mathcal{O}\left(\epsilon\te,\epsilon^2,\te^2\right)\\
\mathbb{E}(G_i)^2&=\epsilon^{d}m^2\tmo^{-2} \left(1- 2\te \tm\tom-2\te\tm q^{-1}\Delta q+2\epsilon m^{-1} \omega\right)+\epsilon^{d/2}\mathcal{O}\left(\epsilon\te,\epsilon^2,\te^2\right),\\
\mathbb{E}(F_i^2)&=2^{-d/2}\epsilon^{d/2}G_{q,\epsilon/2}\left(u^2q_{\te}^{-2}\right)\\
&=2^{-d/2}\epsilon^{d/2}m\tmo^{-2} u^2q^{-1}\bigg(1-\te\tm\tom -\te\tm q^{-1}\Delta q+(\epsilon/2) m^{-1}\omega \\
&\hspace{.25 in}+\epsilon\frac{\ko(u^2q^{-1})}{2u^2q^{-1}}\bigg)+2^{-d/2}\epsilon^{d/2}\mathcal{O}\left(\epsilon\te,\epsilon^2,\te^2\right)\\
\mathbb{E}(G_i^2)&=2^{-d/2}\epsilon^{d/2}G_{q,\epsilon/2}\left(q_{\te}^{-2}\right)\\
&=2^{-d/2}\epsilon^{d/2}m\tmo^{-2} q^{-1}\bigg(1- \te \tm\tom-\te\tm q^{-1}\Delta q+(\epsilon/2) m^{-1} \omega+\epsilon \frac{\ko(q^{-1})}{2q^{-1}}\bigg)\\
&\hspace{.25 in}+2^{-d/2}\epsilon^{d/2}\mathcal{O}\left(\epsilon\te,\epsilon^2,\te^2\right)\\
\mathbb{E}(G_iF_i)&=2^{-d/2}\epsilon^{d/2}G_{q,\epsilon/2}\left(uq_{\te}^{-2}\right)\\
&=2^{-d/2}\epsilon^{d/2}m\tmo^{-2} uq^{-1}\bigg(1-\te \tm\tom -\te\tm q^{-1}\Delta q+(\epsilon/2) m^{-1}\omega +\epsilon\frac{\ko(uq^{-1})}{2uq^{-1}}\bigg)\\
&\hspace{.25 in}+2^{-d/2}\epsilon^{d/2}\mathcal{O}\left(\epsilon\te,\epsilon^2,\te^2\right).
\end{align*}
Since $\mathbb{E}(Y_j)=0$, 
\begin{align*}
\mathrm{Var}(Y_j)&=\mathbb{E}(Y_j^2)\\
&= \mathbb{E}(G_i)^2\mathbb{E}(F_i^2)+\mathbb{E}(F_i)^2\mathbb{E}(G_i^2)-2\mathbb{E}(G_i)\mathbb{E}(F_i)\mathbb{E}(F_iG_i)\\ &\hspace{.25 in}+a\epsilon^{3d/2}\mathcal{O}(\epsilon^2)).
\end{align*}
Proceed with the calculation, 
\begin{align*}
\mathbb{E}(G_i)^2\mathbb{E}(F_i^2)&=2^{-d/2}\epsilon^{3d/2}m^3\tmo^{-4}u^2q^{-1}\bigg(1-3\te\tm\tom-3\te\tm q^{-1}\Delta q\\ &\hspace{.25 in}+\frac{5}{2}\epsilon m^{-1}\omega+\epsilon\frac{\ko u^2q^{-1}}{2u^2q^{-1}}\bigg)+\epsilon^{3d/2}\mathcal{O}(\epsilon\te,\epsilon^2,\te^2)\\
\mathbb{E}(F_i)^2\mathbb{E}(G_i^2)&=2^{-d/2}\epsilon^{3d/2}m^3\tmo^{-4}u^2q^{-1}\bigg(1-3\te\tm\tom-3\te\tm q^{-1}\Delta q \\ &\hspace{.25 in}+\frac{5}{2}\epsilon m^{-1}\omega+ 2\epsilon u^{-1}\ko u+(\epsilon/2)q\ko q^{-1}\bigg)\\ &\hspace{.25 in}+\epsilon^{3d/2}\mathcal{O}(\epsilon\te,\epsilon^2,\te^2)\\
\mathbb{E}(F_i)\mathbb{E}(G_i)\mathbb{E}(F_iG_i)&=2^{-d/2}\epsilon^{3d/2}m^3\tmo^{-4}u^2q^{-1}\bigg(1-3\te\tm\tom-3\te\tm q^{-1}\Delta q\\&\hspace{.25 in}+\frac{5}{2}\epsilon m^{-1}\omega+u^{-1}\ko u +\epsilon \frac{\ko uq^{-1}}{2uq^{-1}}\bigg)+\epsilon^{3d/2}\mathcal{O}(\epsilon\te,\epsilon^2,\te^2),
\end{align*}
we have,
\BEA
\mathrm{Var}(Y_j)&=&2^{-d/2}\epsilon^{3d/2+1}m^3\tmo^{-4}u^2q^{-1}\bigg(\frac{\ko(u^2 q^{-1})}{2u^2q^{-1}}+\frac{\ko q^{-1}}{2q^{-1}}-\frac{\ko(uq^{-1})}{uq^{-1}}\bigg)\nonumber\\ &&\hspace{.25in}+ \epsilon^{3d/2}\mathcal{O}(\epsilon\te, \epsilon^2,\te^2)\nonumber\\
&=&2^{-d/2-1}\epsilon^{3d/2+1}m^3\tmo^{-4}\bigg(\ko(u^2q^{-1})+u^2\ko q^{-1}-2u\ko(uq^{-1})\bigg)\nonumber\\ &&\hspace{.25in}+ \epsilon^{3d/2}\mathcal{O}(\epsilon\te, \epsilon^2,\te^2).\label{tempvar}
\EEA

Recall that $\ko=b\cdot \nabla+\frac{1}{2}C_{ij}\nabla_i\nabla_j$, where the dot product, gradient and covariant derivatives are all defined with respect to Riemannian metric, $g$, inherited by $\mathcal{M}$ from the ambient space. From lemma $4.2$ in \cite{bs:16}, the hessian term can be written as $$C_{ij}\nabla_i\nabla_j =\Delta_{\tilde{g}}+\kappa \cdot \nabla,$$
where $\kappa$ depends on $C$. Here the Laplacian is defined with respect to a modified metric $\tilde{g}(u,v)=g(c^{-1/2}u,c^{-1/2}v)$ for all $u,v\in T_{x}\mathcal{M}$, whereas the advection term is defined with respect to $g$. Therefore,
\BEA
\ko=(b+\kappa)\cdot \nabla+\Delta_{\tilde{g}}.\nonumber
\EEA
Also, using the fact that
\BEA
\Delta_{\tilde{g}} (uv) =  u \Delta_{\tilde{g}} v + v \Delta_{\tilde{g}} u + \nabla_{\tilde{g}} u \cdot  \nabla_{\tilde{g}} v,\nonumber
\EEA 
one can deduce that the terms in the bracket in \eqref{tempvar} is,
\BEA
\ko(u^2q^{-1})+u^2\ko q^{-1}-2u\ko(uq^{-1}) = q^{-1}\|\nabla_{\tilde{g}}u\|^2,\nonumber
\EEA
where all the gradient terms vanish. Substituting this to \eqref{tempvar}, we have
\BEA
\mathrm{Var}(Y_j) =2^{-d/2-1}\epsilon^{3d/2+1}m^3\tmo^{-4}q^{-1}\|\nabla_{\tilde{g}}u\|^2+\epsilon^{3d/2}\mathcal{O}(\epsilon\te, \epsilon^2,\te^2).\nonumber
\EEA
Also, since 
\BEA
\mathbb{E}(G_i)^4 = \epsilon^{2d}m^4\tmo^{-4} + \epsilon^{2d}\mathcal{O}\left(\epsilon,\te\right)\nonumber,
\EEA
the inequality in \eqref{chernoff} becomes,
\begin{align*}
P\left(\sum_{j \neq i}Y_j>a\epsilon(N-1)\mathbb{E}(G_i)^2\right)\leq 2\exp\bigg(-\frac{a^2\epsilon^{d/2+1}(N-1)m}{2^{-d/2-1}\|\nabla_{\tilde{g}}u\|^2q^{-1}} \bigg).
\end{align*}
This inequality basically means that the error in estimating $L_{\epsilon,\te}$ with $\hat{L}_{\epsilon,\te}$ (from the relation in \eqref{discreteerrorLee}) is of order,
\BEA
a = \mathcal{O}\left(\frac{\|\nabla_{\tilde{g}}u\| q^{-1/2}}{\sqrt{N}\epsilon^{d/4+1/2}}\right),\nonumber
\EEA
which is much larger than the error of neglecting the $j$th term, $\mathcal{O}(N^{-1}\te^{-d/2})$, in the summation in \eqref{discreteerrorLee}. This completes the proof of Lemma~\ref{lemma5}.


\section*{References}

\begin{thebibliography}{10}
\expandafter\ifx\csname url\endcsname\relax
  \def\url#1{\texttt{#1}}\fi
\expandafter\ifx\csname urlprefix\endcsname\relax\def\urlprefix{URL }\fi
\expandafter\ifx\csname href\endcsname\relax
  \def\href#1#2{#2} \def\path#1{#1}\fi

\bibitem{evans1998partial}
L.~Evans, A.~M. Society, Partial Differential Equations, Graduate studies in
  mathematics, American Mathematical Society, 1998.

\bibitem{Feynman}
R.~P. Feynman, R.~B. Leighton, M.~Sands, The Feynman lectures on physics, Vol.
  I: The new millennium edition: mainly mechanics, radiation, and heat, Vol.~1,
  Basic books, 2011.

\bibitem{Neutron}
E.~L. Wachspress, Iterative solution of elliptic systems,: And applications to
  the neutron diffusion equations of reactor physics, Prentice-Hall, 1966.

\bibitem{MoPer}
P.~M\"orters, Y.~Peres, Brownian Motion, Cambridge University Press, 2010.

\bibitem{dziuk2013finite}
G.~Dziuk, C.~M. Elliott, Finite element methods for surface pdes, Acta Numerica
  22 (2013) 289--396.

\bibitem{camacho}
F.~Camacho, A.~Demlow, L2 and pointwise a posteriori error estimates for fem
  for elliptic pdes on surfaces, IMA Journal of Numerical Analysis 35~(3)
  (2015) 1199--1227.

\bibitem{bonito2016high}
A.~Bonito, J.~M. Casc{\'o}n, K.~Mekchay, P.~Morin, R.~H. Nochetto, High-order
  afem for the laplace--beltrami operator: Convergence rates, Foundations of
  Computational Mathematics 16~(6) (2016) 1473--1539.

\bibitem{bertalmio2001variational}
M.~Bertalm{\i}o, L.-T. Cheng, S.~Osher, G.~Sapiro, Variational problems and
  partial differential equations on implicit surfaces, Journal of Computational
  Physics 174~(2) (2001) 759--780.

\bibitem{memoli2004implicit}
F.~M\'emoli, G.~Sapiro, P.~Thompson, Implicit brain imaging, NeuroImage 23
  (2004) S179--S188.

\bibitem{ruuth2008simple}
S.~J. Ruuth, B.~Merriman, A simple embedding method for solving partial
  differential equations on surfaces, Journal of Computational Physics 227~(3)
  (2008) 1943--1961.

\bibitem{piret2012orthogonal}
C.~Piret, The orthogonal gradients method: A radial basis functions method for
  solving partial differential equations on arbitrary surfaces, Journal of
  Computational Physics 231~(14) (2012) 4662--4675.

\bibitem{li2016convergent}
Z.~Li, Z.~Shi, A convergent point integral method for isotropic elliptic
  equations on a point cloud, Multiscale Modeling \& Simulation 14~(2) (2016)
  874--905.

\bibitem{cl:06}
R.~Coifman, S.~Lafon, Diffusion maps, Appl. Comput. Harmon. Anal. 21 (2006)
  5--30.

\bibitem{bh:15vb}
T.~Berry, J.~Harlim, Variable bandwidth diffusion kernels, Appl. Comput.
  Harmon. Anal. 40 (2016) 68--96.

\bibitem{li2017point}
Z.~Li, Z.~Shi, J.~Sun, Point integral method for solving poisson-type equations
  on manifolds from point clouds with convergence guarantees, Communications in
  Computational Physics 22~(1) (2017) 228--258.

\bibitem{bs:16}
T.~Berry, T.~Sauer, Local kernels and the geometric structure of data, Applied
  and Computational Harmonic Analysis 40~(3) (2016) 439--469.

\bibitem{bressan2013lecture}
A.~Bressan, Lecture Notes on Functional Analysis: With Applications to Linear
  Partial Differential Equations, Graduate studies in mathematics, American
  Mathematical Society, 2013.

\bibitem{lonseth1954}
A.~T. Lonseth, Approximate solutions of fredholm-type integral equations, Bull.
  Amer. Math. Soc. 60~(5) (1954) 415--430.

\bibitem{Atkinson1967}
K.~E. Atkinson, The solution of non-unique linear integral equations,
  Numerische Mathematik 10~(2) (1967) 117--124.

\bibitem{KAMMERER1972547}
W.~Kammerer, M.~Nashed, Iterative methods for best approximate solutions of
  linear integral equations of the first and second kinds, Journal of
  Mathematical Analysis and Applications 40~(3) (1972) 547 -- 573.

\bibitem{Harlim}
J.~Harlim, Data-Driven Computational Methods: Parameter and Operator
  Estimations, Cambridge University Press, 2018.

\bibitem{IDM}
T.~Berry, J.~Harlim, Iterated diffusion maps for feature identification,
  Applied and Computational Harmonic Analysis 45~(1) (2018) 84 -- 119.

\bibitem{singer2006graph}
A.~Singer, From graph to manifold laplacian: The convergence rate, Applied and
  Computational Harmonic Analysis 21~(1) (2006) 128--134.

\bibitem{ahlberg1963}
J.~Ahlberg, E.~Nilson, Convergence properties of the spline fit, Journal of the
  Society for Industrial and Applied Mathematics 11~(1) (1963) 95--104.

\bibitem{varah1975}
J.~M. Varah, A lower bound for the smallest singular value of a matrix, Linear
  Algebra and its Applications 11~(1) (1975) 3--5.

\bibitem{fornberg2015solving}
B.~Fornberg, N.~Flyer, Solving pdes with radial basis functions, Acta Numerica
  24 (2015) 215--258.

\bibitem{sarra:17}
S.~Sarra, The matlab radial basis function toolbox, Journal of Open Research
  Software 5~(1) (2017) 8.

\bibitem{crane}
K.~Crane, Keenan's 3d model repository (2018),
  \url{http://www.cs.cmu.edu/~kmcrane/Projects/ModelRepository}.

\bibitem{Walker}
S.~Walker, FELICITY: A Matlab/C++ Toolbox for Developing Finite Element Methods and Simulation Modeling, SIAM Journal on Scientific Computing, 40~(2) (2018) C234--C257, \url{https://doi.org/10.1137/17M1128745}.

\end{thebibliography}
\end{document}